\documentclass[a4paper,11pt]{article}

\usepackage{subfigure}
\usepackage{amsfonts}
\usepackage{amssymb}
\usepackage{amsthm}
\usepackage{amsmath}
\usepackage{a4wide}
\usepackage{mathrsfs}
\usepackage{epsfig}
\usepackage{tikz,fp,ifthen}
\usepackage{float}
\usepackage[all]{xy}

\newcommand{\pdfgraphics}{\ifpdf\DeclareGraphicsExtensions{.pdf,.jpg}\else\fi}
\usepackage{graphicx}

\usepackage{color}
\definecolor{hanblue}{rgb}{0.27, 0.42, 0.81}
\definecolor{red}{rgb}{1.0, 0.0, 0.0}
\usepackage[colorlinks, citecolor=hanblue,linkcolor=red]{hyperref}

\usepackage{mathtools}
\usepackage{hyperref}

\newcommand{\res}{\mathop{\hbox{\vrule height 7pt width .5pt depth 0pt
\vrule height .5pt width 6pt depth 0pt}}\nolimits}
\renewcommand{\div}{{\rm div}\,}

\newcommand{\R}{\mathbb{R}}

\newcommand{\N}{\mathbb{N}}

\newcommand{\Ha}{\mathcal{H}}

\newcommand\restr[2]{{
  \left.\kern-\nulldelimiterspace 
  #1 
  \vphantom{\big|} 
  \right|_{#2} 
  }}

\numberwithin{equation}{section}

\theoremstyle{plain}

\newtheorem{teo}{Theorem}[section]
\newtheorem{prob}[teo]{Problem}
\newtheorem{lemma}[teo]{Lemma}
\newtheorem{prop}[teo]{Proposition}
\newtheorem{cor}[teo]{Corollary}

\theoremstyle{definition}
\newtheorem{dfnz}[teo]{Definition}

\theoremstyle{remark}
\newtheorem{rem}[teo]{Remark}

\newtheorem{ex}[teo]{Example}

\begin{document}
\pdfgraphics 

\pdfgraphics 

\title{Calibrations for minimal networks in a covering space setting}

\author{
Marcello Carioni \footnote{Institut f\"ur Mathematik, Karl--Franzens--Universit\"at, Heinrichstrasse 36, 8010, Graz,
Austria}
\and 
Alessandra Pluda \footnote{Dipartimento di Matematica, Universit\`a di Pisa, Largo Bruno Pontecorvo 5, 56127, Pisa, Italy}}
\date{}

\maketitle

\begin{abstract}
\noindent
In this paper we define a notion of calibration for an 
approach  
to the classical Steiner problem in a covering space setting and we give some explicit examples.
Moreover we introduce the notion of calibration in families:
the idea is to divide the set of competitors in a suitable way,
defining an appropriate (and weaker) notion of calibration.
Then, calibrating the candidate minimizers
 in each family and comparing their perimeter, it is possible to find the minimizers of the minimization problem. 
Thanks to this procedure we prove the minimality of the Steiner configurations spanning  
the vertices of a regular hexagon
and of a regular pentagon.
\end{abstract}

\date{}

\maketitle

\begin{center}

\vspace*{1mm}

\end{center}

\section{Introduction}
The classical Steiner problem, whose first modern formulation can be found in~\cite{courant}, 
can be stated as follows: given a collection $S$ of $m$ points of $\R^n$, 
find the connected set that contains $S$ with minimal length, namely
\begin{equation}\label{ste}
\inf  \{\Ha^1(K) : K \subset \R^n, \mbox{ connected and such that } S \subset K\}\,
\end{equation}
(we refer to~\cite{Ivan} for a survey on the topic).
In its highest generality the problem can be stated replacing the ambient space $\R^n$
with any metric space~\cite{paopaostepanov}.
From a computational point of view 
the Steiner problem is NP-hard, hence,
it is interesting to attack it
in new ways 
(for example by 
 multiphase 
approximations \`{a} la Modica--Mortola~\cite{belpaover, orlandi, sant2, chambolle, sant})
in order to improve
the rate of convergence of the 
algorithms.
From a theoretical perspective
the solution of the Steiner problem 
by variational methods 
has received an increasing interest
starting from several results of the 90's, 
which establish an equivalence between the Steiner problem and
the minimal partition problem~\cite{ab1,ab2, morel, tamanini}, and 
ending to more recent approaches, such as
currents or vector valued BV functions defined on a covering space~\cite{cover, brakke},
currents with coefficients in a group~\cite{annalisaandrea}, 
rank-one tensor valued measures~\cite{orlandi}.

\medskip

The first approach via covering space to the Steiner problem, and more in general to Plateau's type problems,
is due to Brakke~\cite{brakke}.
Having in mind a possible candidate minimizer
(unoriented soap films, soap films with singularities, soap films touching only part of a knotted curve)
for a certain Plateau's type problem 
he introduces a double covering called \emph{pair covering space}
chosen compatibly to the soap film he wants to obtain as a minimizer.
Then he minimizes the mass of integral currents defined on it,
and only in some special cases, via a calibration argument, he proves
that the minimizer coincides with his candidate.
As a consequence of this \emph{ad hoc} approach, 
he did not give
an explicit proof of the equivalence 
with the Steiner problem; 
however his setting allows to describe    
a great variety of different objects, for istance soap films in higher dimension. 
In~\cite{cover} the authors revive Brakke's covering space approach, 
constructing an $m$--sheeted covering space 
$Y_\Sigma$ of $\mathbb{R}^2\setminus S$ and minimizing the total variation 
of vector valued $BV$ functions on  $Y_\Sigma$ satisfying a certain constraint. 
In this setting the authors prove the equivalence between their minimization problem
and the Steiner problem in the plane.

\medskip

The first part of this paper is devoted to an improvement of the  
result in~\cite{cover} reducing the vector valued problem in~\cite{cover} to a scalar one: we minimize the \emph{perimeter} 
in a family of finite perimeter sets in $Y_\Sigma$
instead of the total variation of constrained vector valued $BV$ functions.
In particular we state the following minimization problem in $Y_\Sigma$:
\begin{equation}\label{minpro2}
\mathscr{A}_{constr}(S) = \inf \left\{P(E) : E\in\mathscr{P}_{constr}(Y_\Sigma)\right\}\,,
\end{equation}
where $\mathscr{P}_{constr}(Y_\Sigma)$ is the space of 
sets of finite perimeter in $Y_\Sigma$ 
satisfying a suitable  constraint.
Once proved the existence of minimizers
we show an equivalence between our minimization problem
and the classical Steiner problem (we refer the reader to the beginning of Section~\ref{equivalence}
for further explanations).

\medskip

In the second part of the paper we introduce a notion of calibration suitable to our setting.
In the context of minimal surfaces, a \emph{calibration} 
for a $k$--dimensional oriented manifold in $\R^{n+1}$
is a closed $k$-form $\omega$ such that $\vert\omega\vert\leq 1$ (the so--called size condition) and $\langle \omega, \xi\rangle = 1$, where $\xi$ is the unit $k$-vector orienting the manifold.
It is easy to see that the existence of a calibration for a certain manifold
implies that the manifold is area minimizing in its homology class. 
In Definition~\ref{caliconvering}
we adapt this notion to our setting taking advantage of the theory of Null-Lagrangians \cite{ab1, ab2}: a calibration for $E \in \mathscr{P}_{constr}(Y_\Sigma)$ is a divergence-free vector field defined on the covering space $Y_\Sigma$ 
such that $\int_{Y_\Sigma} \Phi \cdot D\chi_E = P(E)$ and
a suitable size condition for $\Phi$ is fulfilled (see Remark \ref{twoinsteadofone}). 
As for
minimal surfaces, we show that the existence of a calibration for a set $E\in \mathscr{P}_{constr}(Y_\Sigma)$ implies
that it is a minimizer 
 to~\eqref{minpro2}. 
In order to show the usefulness of our theory we give the explicit examples of calibrations for the Steiner minimal configuration connecting two points, 
three points located at the vertices of any triangle and for the four vertices of a square.

\medskip

A notion of calibration for the partition problem  
was firstly introduced by Morgan and Lawlor in~\cite{lawmor}:
their \emph{paired calibration} technique allows 
to prove the minimality of soap films among all the competitors 
that split the domain in a fixed number of regions.
In the context of the Steiner problem
the limit of this approach is that 
it can be applied only when
the points of $S$ 
belong to
the boundary of a convex set.
As mentioned previously,
a notion of calibration, adapted to the covering space approach,  is proposed also in~\cite{brakke}.
Finally, Marchese and Massaccesi in~\cite{annalisaandrea} 
describe Steiner trees using currents with coefficient in groups 
and rephrase the Steiner problem as a minimum problem for the mass of currents.
In this way they were able to introduce a related notion of calibration  (see also \cite{coveringbis} for a comparison of the different notions).
Both~\cite{brakke} and~\cite{annalisaandrea}  have a companion paper devoted
to numerical results (see~\cite{brakke2} and~\cite{MasOudVel16}).
We underline that our approach, 
as the one introduced in~\cite{brakke, annalisaandrea}, 
does not require that the points of $S$ lie on the boundary of a convex set.

\medskip

The goal of the last part of the paper is to tackle the minimality of the Steiner minimal configurations for the vertices of a regular pentagon and of a regular hexagon using the theory of calibrations.

We remind that the explicit minimizers for the Steiner problem, if the points of $S$ are
the vertices of regular $n$-gon, are well known.
In particular for
$n\geq 6$ the Steiner minimal network is the polygon without an edge.
The first proof for the cases $n\leq 6$ and $n\geq 13$
is due to Jarnik and K\"ossler in 1934~\cite{jarnik}. 
Fifty years later,
Du, Hwang and Weng proved
the remaining cases~\cite{Hwang}.

\medskip

The main purpose of
searching for a calibration
is to have an easy argument to show the minimality of a certain candidate.
Unfortunately 
finding a calibration is in general not an easy task
and only very few example are known even for
the Steiner problem. 
As already anticipated, in this work we propose a calibration argument to prove in a strikingly easy way the minimality for the Steiner minimal configuration connecting the vertices of the regular pentagon and of the regular hexagon. 

The interest of our technique goes beyond these specific results because it can be generalized to
arbitrary configurations of points in $\R^2$ and suggests how 
to ``decompose" the Steiner problem in several simpler convex problems that can be solved (and calibrated) separately. 
Moreover the authors believe that a similar idea could be applied to different variational problems.
The idea is to divide the set of competitors 
in different families, denoted by $\mathcal{F}(\mathcal{J})$,
and define an appropriate notion of calibration in each family
with a  weaker size condition.
To be more precise all the competitors that belong to the same family
share a property related to the projection of their essential boundary onto the base set $M$:
for certain couples of indices $(i,j)$ in $\{1,\ldots,m\}\times\{1,\ldots,m\}$
the intersection of the projections onto $M$ of the boundary of the part of
the set $E$ 
in the $i$--th sheet and in the $j$--th sheet is $\mathcal{H}^1$--negligible.
As a consequence,  the definition of calibration in a family does not require to
verify the size condition for the pairs of sheets associated to these couples of indices $(i,j)$.
Once identified a candidate minimizer in each family, 
we calibrate them
and in conclusion we compare their energy
to find the explicit global 
minimizers of Problem~\eqref{minpro2}. 
Thanks to this procedure we prove the minimality of the Steiner configurations spanning  
the vertices of a regular hexagon (the hexagon without one edge)
and of the regular pentagon.

\medskip

Finally we outline the structure of the paper:
in the beginning of Section~\ref{equivalence}
we summarize the setting introduced in~\cite{cover} 
describing the construction of the covering space and we define finite perimeter sets on it. 
In Subsection~\ref{problem} we introduce the space $\mathscr{P}_{constr}(Y_\Sigma)$ and 
prove the existence of minimizers for Problem~\eqref{minpro2}.
Then, in Theorem~\ref{regularity}, we present a regularity result for the essential boundary of minimizers 
proving a \emph{local} equivalence 
with the problem of minimal partitions~\cite{ab1,ab2, morel, tamanini}.
We conclude Section~\ref{equivalence} proving the equivalence between 
the classical Steiner problem and our minimization problem~\eqref{minpro2}.

In Section~\ref{seccal}, after giving the definition of calibration, 
we show that the existence of a calibration for $E \in \mathscr{P}_{constr}(Y_\Sigma)$ 
implies minimality of $E$ with respect to~\eqref{minpro2}. We construct explicit examples of  calibration for the Steiner minimal configuration connecting two points, three points and the four vertices of a square.

In Section~\ref{famiglie} we develop the notion of calibration in families
and use this tool to prove 
the minimality of the Steiner minimal configurations spanning  
the vertices of a regular hexagon and of a regular pentagon.

\subsection*{Acknowledgements}
The authors are warmly grateful to Giovanni Bellettini and Maurizio Paolini for several discussions and helpful conversations.
Once this manuscript was already completed, 
the authors have been 
told by Frank Morgan
that the idea about separating the competitors into 
classes for the hexagon problem was found independently
by Gary Lawlor and presented at Lehigh University in 1998, but was not published 
even in preprint form.
The authors would like to thank 
Frank Morgan and Gary Lawlor for their remarks about the history of the problem and for their kind mail correspondence.

\section{Equivalence between Problem~\eqref{minpro2} and the Steiner problem}\label{equivalence}

In this section
we consider the construction of the $m$--sheeted covering space 
$Y_\Sigma$ of $M:=\mathbb{R}^2\setminus S$ presented in~\cite{cover}
and we define sets of finite perimeter on the covering.
In particular we define the space
$\mathscr{P}_{constr}(Y_\Sigma)$ of
the sets of finite perimeter $E$ in $Y_\Sigma$ 
satisfying a suitable  boundary condition at infinity 
and
such that 
for almost every $x$ in the base space there exists exactly one point $y$ of $E$
such that $p(y)=x$, where $p$ is the projection onto the base space 
(see Definition~\ref{zerounoconstr}).
Notice that for every set $E\in \mathscr{P}_{constr}(Y_\Sigma)$ 
it is possible to show a 
 formula (Proposition \ref{dueh}) that relates its perimeter 
to the $\mathcal{H}^1$-measure of the projection onto the base space $M$
of its essential boundary.
In Subsection \ref{problem} we state the minimization problem~\eqref{minpro2} 
and we prove existence and non--triviality of the minimizers. 


The last part of the section is devoted to the
equivalence between our minimization problem
and the classical Steiner problem.
On one side it is enough to show that given a minimizer
for Problem~\eqref{minpro2}, the network
obtained as the closure of the projection onto $M$
of the essential boundary is a competitor for the Steiner problem.
Roughly speaking 
given the $m$ points of $S$ in $\mathbb{R}^2$,
the covering space of $M$
is constructed in such a way that  
if we consider
a closed curve $\gamma$ (namely a loop) with index one with respect to at most $m-1$
points and with index zero with respect to at least one point of $S$, then
$p^{-1}(\gamma)$ is connected.
Combining this property of the covering with the constraint on the set
it is possible to show that the set 
$S$ is contained in a connected component of 
the closure of the projection of the essential boundary. 
On the other hand  we describe a procedure to construct
a set in $\mathscr{P}_{constr}(Y_\Sigma)$ from a minimal Steiner graph.


\subsection{Construction of the covering space}\label{setcov}

For the rest of the paper we consider $S = \{p_1,\ldots,p_m\}$  a finite set of points
in $\mathbb{R}^2$ and $M \coloneqq \R^2 \setminus S$. Moreover we fix an open, 
smooth and bounded set $\Omega \subset \R^2$ such that $\mbox{Conv}(S)\subset \Omega$, where $\mbox{Conv}(S)$ is the convex envelope of $S$.

\begin{dfnz}[Admissible cuts]\label{admissible}
We denote by 
$\mbox{Cuts}(S)$ the set of all 
$\Sigma \coloneqq \bigcup_{i=1}^{m-1} \Sigma_i \subset \Omega$ such that:
\begin{itemize}
\item[(a)] for $i=1,\ldots,m-1$, $\Sigma_i$ is a Lipschitz simple curve starting at $p_i$ and ending at $p_{i+1}$;
\item[(b)] if $m>2$ then $\Sigma_i \cap \Sigma_{i+1}=\{p_{i+1}\}$ for $i=1,\ldots,m-2$;
\item[(c)] $\Sigma_i \cap \Sigma_l = \emptyset$ for any $i,l=1,\ldots,m-1$ such that $|l - i|>1$.
\end{itemize}
Moreover we denote by $\mbox{\textbf{Cuts}}(S)$ the set of all pairs ${\bf{\Sigma}} \coloneqq (\Sigma,\Sigma')$ such that
\begin{itemize}
\item[(i)] $\Sigma,\Sigma' \in \mbox{Cuts}(S)$ and $\Sigma\cap \Sigma' = S$;
\item[(ii)] if $m>2$, for every  $i=2,\ldots,m-1$  let $C_i(p_i,\varepsilon)$ be
a circle centered at $p_i$ with radius $\varepsilon$ such that $C_i(p_i,\varepsilon)\cap \Sigma_{i-1}\neq\emptyset$ and $C_i(p_i,\varepsilon)\cap \Sigma_{i}\neq\emptyset$.
Denote by $x_i$ (resp $y_i$) the intersection of $C_i$ with $\Sigma_{i-1}$ (resp. with $\Sigma_{i}$).
Then there exists an arc of $C_i$ connecting $x_i$ and $y_i$ and not intersecting $\Sigma'$. 
\end{itemize}
\end{dfnz}

Fix ${\bf\Sigma}=(\Sigma,\Sigma') \in \mbox{\textbf{Cuts}}(S)$ and define
\begin{equation*}
D\coloneqq\mathbb{R}^2 \setminus \Sigma\,, \quad\quad D'\coloneqq\mathbb{R}^2 \setminus \Sigma'\,
\end{equation*}
and
\begin{equation*}
X = \bigcup_{j=1}^m(D,j)\;\, \cup \bigcup_{j'=m+1}^{2m} (D',j')
\end{equation*}
that is the space made of $m$ disjoint copies of $D$ and of $D'$.

\begin{figure}[H]
\centering
\begin{tikzpicture}[scale=0.7]
\node[inner sep=0pt] at (0,0)
    {\includegraphics[width=0.6\textwidth]{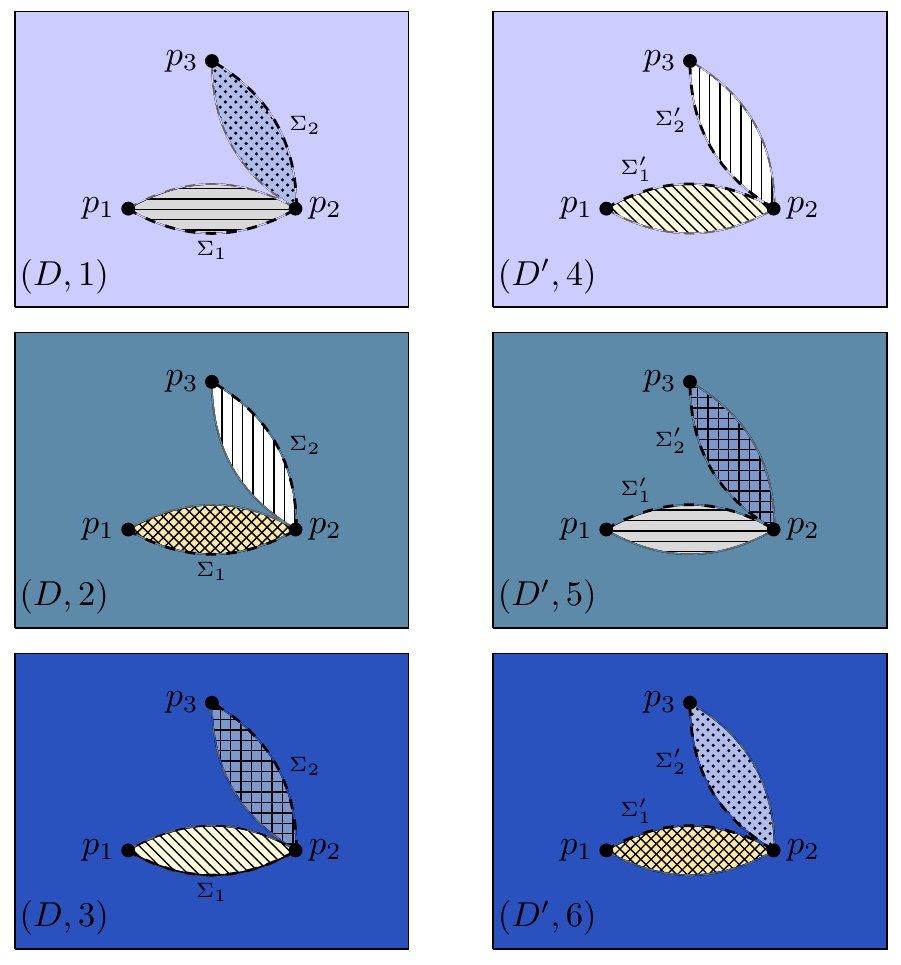}};
\end{tikzpicture}
\caption{An explicit construction of $X$ when $m=3$ and $S=\{p_1,p_2,p_3\}$. 
The regions identified 
by the equivalence relation $\sim$ are represented with the same pattern and color.}
\end{figure}

Let $I_i$ be the open, bounded set enclosed by $\Sigma_i$ and $\Sigma_i'$ and $O = \mathbb{R}^2 \setminus \bigcup_{i=1}^{m-1} \overline{I_i}$.
Given $(x,j) \in (D,j)$ with $j\in\{1,\ldots,m\}$ and $(x',j') \in (D',j')$ with $j'\in\{m+1,\ldots,2m\}$, 
we define the equivalence relation $\sim$ in $X$ as 
$(x,j)\sim (x',j')$ if and only if one of the following conditions holds:
\begin{equation}\label{ide2}
\left\{\begin{array}{ll}
j\equiv j'\ (\mbox{mod}\ m),& x=x' \in O\,, \\
j\equiv j'-i\ (\mbox{mod}\ m), & x=x'\in I_i,\ i=1,\ldots , m-1\,.
\end{array}
\right.
\end{equation}

\begin{dfnz}
We define  $Y_\Sigma$ to be  
the topological quotient space induced by $\sim$, i.e.
$$
Y_\Sigma\coloneqq X \Big{/}\sim\,.
$$
\end{dfnz}
Finally we denote by $\tilde \pi: X \rightarrow Y_\Sigma$ the projection induced by the equivalence relation,
by $\pi$ the projection from $X$ to the space $M$ 
and by $p: Y_\Sigma \rightarrow M$ the map that makes the following diagram commutative:
\begin{equation*}
\xymatrix{
X  \ar[dr]_{\pi} \ar[r]^{\widetilde{\pi}} & Y_{\Sigma} \ar[d]^{p}\\
& M}
\end{equation*} 

\begin{prop}
The map $p: Y_\Sigma \rightarrow M$ is well-defined and the pair $(Y_\Sigma, p)$
is a covering space of $M$.
\end{prop}

\begin{rem}
The covering $(Y_\Sigma,p)$ of $M$ inherits the local structure of $M$, 
hence, in our case, it is a Riemannian manifold.
\end{rem}
In order to be able to work on Euclidean spaces we define 
the following natural local parametri-\\
zations of $Y_\Sigma$.
\begin{dfnz}[Local parametrizations]\label{local}
For every $j=1,\ldots\,m$
we define the local parame-\\
trizations $\psi_j:D\to\widetilde{\pi}\left((D,j )\right)$ 
as
$$
\psi_j(x)\coloneqq \widetilde{\pi}\left(\left(x,j\right)\right)\;\;\text{for every}\;x\in D\,,
$$
or equivalently as
$$
\psi_j(x)\coloneqq \left(\restr{p}{\tilde \pi ((D,j))}\right)^{-1}(x)\;\;\text{for every}\;x\in D\,.
$$
The local parametrizations $\psi_{j'}:D'\to\widetilde{\pi}\left((D',j' )\right)$, 
for $j'=m+1,\ldots\,2m$ are analogously defined. 
\end{dfnz}

%
%
%
%
%

\subsection{Sets of finite perimeter on the covering space}

It is natural to endow the space $Y_\Sigma$ with a measure $\mu$ defined as the sum on every sheet of the covering space of the push-forward by the local parametrization introduced in Definition~\ref{local} of the Lebesgue measure $\mathscr{L}^2$. Given a Borel set $E\subset Y_\Sigma$ we define
\begin{equation*}
\mu(E) \coloneqq \sum_{j=1}^m \psi_{j\#}\mathscr{L}^2(E\cap \tilde \pi((D,j)))\,.
\end{equation*}

We define the pullback of a function $f$ 
by the local parametrizations $\psi_j$ and $\psi_{j'}$ in the following way:  
\begin{dfnz}\label{pullback}
Consider a function $f:Y_{\Sigma}\to \R^n$. For $j=1,\ldots,m$ and $j'=m+1,\ldots\,2m$ we let 
$f^j:D\to \mathbb{R}^n$, $f^{j'}:D'\to \R^n$ be the maps defined by
\begin{equation*}
f^j\coloneqq f\circ \psi_j \qquad \text{and} \qquad 
f^{j'}\coloneqq f\circ \psi_{j'}\,.
\end{equation*}
\end{dfnz}
By construction, thanks to the identifications given by $\sim$, we have 
\begin{equation}\label{ide}
\left\{\begin{array}{lll}
f^j=f^{j'} 
&\text{in}\; O
&\text{if}\; j\equiv j'\ (\mbox{mod}\ m)\,, \\
f^j=f^{j'} 
&\text{in}\; I_i
&\text{if}\; j\equiv j'-i\ (\mbox{mod}\ m),\ i=1,\ldots , m-1\,.   \\
\end{array}
\right.
\end{equation}

From now on by $\chi_E$ we mean the characteristic function of the set $E$.

\begin{dfnz}
Given a $\mu$--measurable set $E \subset Y_\Sigma$
we define 
\begin{equation*}
E^j:=\{x\in D \,: \, \chi^j_E(x)=1\}=p(E\cap \tilde{\pi}((D,j)))
\end{equation*}
for $j\in\{1,\ldots,m\}$ and  analogously 
\begin{equation*}
E^{j'}:=\{x\in D' \,: \, \chi^{j'}_E(x)=1\}=p(E\cap\tilde{\pi}((D',j')))
\end{equation*}
 for $j'\in\{m+1,\ldots,2m\}$.
\end{dfnz}

\begin{rem}\label{presc}
It is useful sometimes to define functions $f$ (resp. sets $E$)
on the covering space $Y_\Sigma$ prescribing first the parametrizations $f^j : D \rightarrow \R^n$ 
(resp. sets $E^j$)  for every $j=1\ldots,m$ and then deducing $f^{j'}$ in 
$D' \setminus \Sigma$ (resp. sets $E^{j'}$)  for every $j'=m+1,\ldots,2m$, according to \eqref{ide}. 
\end{rem}


We set $L^1(Y_\Sigma) \coloneqq L^1(Y_\Sigma;\R;\mu)$ 
and analogously, $L^1_{loc}(Y_\Sigma) \coloneqq L^1_{loc}(Y_\Sigma;\R;\mu)$. 
We also define the distributional gradient of a function 
$u\in L^1(Y_\Sigma)$ as the linear map
\begin{equation*}
Du(\psi) = -\int_{Y_\Sigma} u\, \div \psi \, d\mu
\end{equation*} 
for $\psi \in C_c^1(Y_\Sigma, \R^2)$, where the space $C_c^1(Y_\Sigma, \R^2)$ is defined in the natural way by the local parametrizations. 
\begin{dfnz}
Given $u\in L^1(Y_\Sigma)$ we say that $u \in BV(Y_\Sigma)$ if $Du$ 
is represented by a Radon measure with bounded total variation.
\end{dfnz}

%

\begin{dfnz}
Given $E\subset Y_\Sigma$ a $\mu$-measurable set and $\Lambda \subset Y_\Sigma$ open, we define
\begin{equation}\label{perimeter}
P(E, \Lambda):=\vert D\chi_{E}\vert (\Lambda) = \sup \left\{\int_E \div\psi \, d\mu : \psi \in C_c^1(\Lambda, \R^2) ,
\|\psi\|_\infty\leq 1\right\}\,.
\end{equation}
We say that $E$ is a set of finite perimeter in $\Lambda$ if $P(E,\Lambda) <\infty$. In the case $E$ is of finite perimeter in $Y_\Sigma$, the definition of $P$ can be extended to all Borel sets $\Lambda \subset Y_\Sigma$ and $\Lambda \rightarrow P(E, \Lambda)$ is a Borel measure in $Y_\Sigma$.
\end{dfnz}

For every $t\in [0,1]$ define the set $E^t \subset Y_\Sigma$ as
\begin{equation*}
E^t = \left\{x\in Y_\Sigma : \lim_{r\rightarrow 0}\frac{\mu(E \cap B_r(x))}{\mu(B_r(x))}=t\right\} \,.
\end{equation*}
We denote by $\partial^{\ast}E$ the essential boundary of $E$ defined as $\partial^{\ast}E = Y_\Sigma \setminus (E^0 \cup E^1)$
(see~\cite[page 158]{afp} for the definition of the essential boundary in the Euclidean setting).

\begin{lemma}[Representation formula for the perimeter and for $D\chi_E$]\label{represent}
Let $\Lambda$ be a Borel set in $Y_\Sigma$ and
$E$ a set of finite perimeter in $Y_\Sigma$.
Then defining $\Lambda^{j'}_\Sigma:= p(\Lambda \cap \tilde \pi((\Sigma \setminus S, j')))$ we have 
\begin{equation}\label{perimetro}
P(E,\Lambda) = \sum_{j=1}^m P(E^j,\Lambda^j)
+ \sum_{j'=m+1}^{2m} P(E^{j'},\Lambda^{j'}_\Sigma)\,
\end{equation}
and
\begin{equation}\label{derivata}
D\chi_E(\Lambda) = \sum_{j=1}^m D\chi_{E^j} (\Lambda^j) 
+ \sum_{j'=m+1}^{2m} D\chi_{E^{j'}}(\Lambda^{j'}_\Sigma)\,.
\end{equation}
\end{lemma}
\begin{proof}
We notice that a Borel set $\Lambda\subset Y_\Sigma$ 
can be decomposed in the union of the disjoint sets:
\begin{equation}\label{declambda}
\Lambda \cap \tilde \pi((D,j)),\ \ j=1,\ldots m\,,\qquad \Lambda \cap \tilde \pi((\Sigma \setminus S, j')),\ \ j'=m+1,\ldots,2m\,.
\end{equation}
Hence it is enough to prove the statements for a $\Lambda \subset \tilde \pi ((D,j))$ for a fixed $j$
or a $\Lambda \subset \tilde \pi ((\Sigma\setminus S,j'))$
for a fixed $j'$.		
Let us assume without loss of generality that $\Lambda \subset \tilde \pi ((D,j))$ is open. 
Consider $\eta \in C^1(\Lambda)$ and compactly supported in $\Lambda$. Then noticing that as $\restr{p}{\Lambda}$ is bijective we have
\begin{align*}
D\chi_E(\eta)=\int_{\Lambda} \chi_E\, \div \eta \, d\mu =
\int_{p(\Lambda \cap E)} \, \div (\eta\circ p^{-1}) \, d\mathscr{L}^2\,=D \chi_{p(E\cap \Lambda)}(\eta\circ p^{-1})\,.
\end{align*} 
Moreover as $\restr{p}{\Lambda}$ is an homeomorphism it is easy to verify that $\eta  \in C_c^1(\Lambda)$ if and only if $\eta \circ p^{-1} \in C_c^1(p(\Lambda))$. Therefore taking the supremum on $\eta$ we have that
\begin{displaymath}
P(E,\Lambda) = P(p(E\cap \Lambda)) = P(p(E\cap \Lambda\cap \tilde \pi((D,j)))) = P(E^j,\Lambda^j).
\end{displaymath} 
If $\Lambda \subset \tilde \pi ((\Sigma\setminus S,j'))$ the decomposition \eqref{derivata} follows performing the same computation in $(D',j')$ and then using the outer regularity of the measure $\Lambda \rightarrow P(E,\Lambda)$.
Formula \eqref{derivata} can be proven similarly.
\end{proof}


\begin{rem}\label{bvpar}
From the computations of the previous lemma, one can easily see that if
$E$ is a set of finite perimeter in $Y_\Sigma$,
then $E^j$ is  a set of finite perimeter in
$D$ for $j=1\,\dots\,m$ (and, respectively, in a similar way one can show that
$E^{j'}$ is  a set of finite perimeter in
$D'$ for $j'=m+1\,\dots\,2m$).
\end{rem}

\subsection{The constrained minimum problem}\label{problem}
We define our simplified version of the minimization problem introduced by
Amato, Bellettini and Paolini in~\cite{cover}, where we use sets of finite perimeter 
instead of vector valued $BV$ functions.


\begin{dfnz}[Constrained sets]\label{zerounoconstr}
We denote by $\mathscr{P}_{constr}(Y_\Sigma)$ the space of
the sets $E$ of finite perimeter in $Y_\Sigma$ such that 
\begin{itemize}\label{zerouno}
\item [$i)$] $\sum_{p(y) = x} \chi_E(y) = 1\ \ $ for almost every $x\in M =\mathbb{R}^2\setminus S$ ,
\item [$ii)$] $\chi_{E^1}(x) = 1\ \  $ for every $x\in \mathbb{R}^2 \setminus \Omega$.
\end{itemize}
\end{dfnz}

In other words a set $E$ of finite perimeter in $Y_\Sigma$
belongs to $\mathscr{P}_{constr}(Y_\Sigma)$ if
for almost every $x$ in the base space there exists exactly one point $y$ of $E$
such that $p(y)=x$.

\begin{rem}
Notice that it is possible to produce different sets in $Y_\Sigma$ (by a permutation of the sheets) satisfying condition $i)$ in Definition \ref{zerounoconstr} that have the same projection on to the base space.
In order to avoid this unpleasant effect we decide to add the condition $ii)$ in Definition \ref{zerounoconstr}.
\end{rem}


%
%
%
%
%
%
\vspace*{2mm}
We state the constrained minimization problem as follows:
\begin{equation*}
\mathscr{A}_{constr}(S) = \inf \left\{P(E) : E\in\mathscr{P}_{constr}(Y_\Sigma)\right\}\,.
\end{equation*}

\begin{rem}\label{independence}
It can be proved as in~\cite{cover} 
(see also~\cite{cover3, cover2})
 that given ${\bf\Sigma}, \hat{{\bf\Sigma}} \in \mbox{\textbf{Cuts}}(S)$ and
$E \in \mathscr{P}_{constr}(Y_\Sigma)$ there exists $\hat{E} \in \mathscr{P}_{constr}(Y_{\hat{\Sigma}})$ such that $p(\partial^\ast E) = p(\partial^\ast \hat E)$.
This implies that the quantity $\mathscr{A}_{constr}(S)$ is independent on the choice of the cuts ${\bf\Sigma}$. 
\end{rem}


\begin{lemma}[Compactness]\label{comp}
Let $(E_n)_{n\in \N}$ be a sequence of sets in $\mathscr{P}_{constr}(Y_\Sigma)$  such that
\begin{equation*}
\sup_{n\in \N} P(E_n) < + \infty\,.
\end{equation*} 
Then there exists $E\in \mathscr{P}_{constr}(Y_\Sigma)$ and 
a subsequence $(E_{n})_{n \in \N}$ (not relabelled) 
converging to $E$ in $L^1(Y_\Sigma)$ as $n\rightarrow +\infty$.
\end{lemma}
\begin{proof}
Notice that, thanks to \eqref{perimetro}, there holds $P(E^j_n, D) \leq P(E_n)< c <\infty$.
Therefore, up to subsequences, $E^j_n\to E^j$ in $L^1(D)$ for every $j=1,\ldots,m$ 
(without relabelling the subsequence). 
The set $E$ is determined according to Remark~\ref{presc}. 
As for every $x\in D$ we have that $p^{-1}(x) = \tilde \pi \circ \pi^{-1}(x) = \bigcup_{j=1}^m \psi_j(x)$, 
the property in Definition~\ref{zerouno} can be rephrased using the local parametrization in the following way: 
\begin{displaymath}
 \sum_{j=1}^m \chi_{E^j_n}(x)=1\,. 
\end{displaymath} 
for almost every $x\in M$.
Hence letting $n \rightarrow +\infty$ we get that $\sum_{i=1}^m \chi_{E^j}(x)=1$ 
for almost every  $x\in M$.
Moreover, as $\chi_{E^1_n}= 1$ in $\mathbb{R}^2 \setminus \Omega$ for every $n$, 
we have $\chi_{E^1}= 1$ in $\mathbb{R}^2 \setminus \Omega$ and therefore 
$E \in \mathscr{P}_{constr}(Y_\Sigma)$.
\end{proof}

\begin{teo}[Existence of minimizers]
There exists a minimizer for Problem \eqref{minpro2}.
\end{teo}
\begin{proof}
The proof follows by the application of the direct method thanks to Lemma $\ref{comp}$ 
and the lower semicontinuity of the perimeter in $L^1_{loc}(Y_\Sigma;\R;\mu)$, that holds because the perimeter defined in \eqref{perimeter} is supremum of continuous functional.
\end{proof}


\subsection{Some properties of the projection of the essential boundary}\label{properties}

\begin{prop}\label{dueh}
Given a set $E$ in $\mathscr{P}_{constr}(Y_\Sigma)$,  we have
\begin{equation*}
P(E) = 2\Ha^1(p(\partial^\ast E))\,.
\end{equation*}
\end{prop} 
\begin{proof}
Consider  a set $E$ in $\mathscr{P}_{constr}(Y_\Sigma)$, 
then thanks to Lemma~\ref{represent} we get
\begin{equation}\label{split}
P(E) = \sum_{j=1}^m P(E^j,D) + \sum_{j'=m+1}^{2m} P(E^{j'},\Sigma \setminus S)\,.  
\end{equation}
Notice firstly that $\partial^\ast E^j$ and $\partial^\ast E^{j'}$ 
are rectifiable sets in $D$ and $D'$ respectively, 
therefore they admit a generalized unit normal that we denote by $\nu_{j}$ and $\nu_{j'}$. 
Define, for $h,k = 1,\ldots,m$, the set
\begin{equation*}
A_{h,k} = \left\{x\in \partial^\ast E^h\cap \partial^\ast E^k : \nu_h(x),\nu_k(x)\; \text{exist and}\;
 \nu_h(x)=-\nu_k(x) \right\}\,,
\end{equation*}
and, for $h',k'=  m+1,\ldots,2m$
\begin{equation*}
A_{h',k'} = \left\{x\in \partial^\ast E^{h'}\cap \partial^\ast E^{k'} : 
\nu_{h'}(x),\nu_{k'}(x)\; \text{exist and}\; \nu_{h'}(x)=-\nu_{k'}(x) \right\}\,.
\end{equation*}
Suppose that $A_{h,k} \subset D$ for every $h,k = 1,\ldots,m$.\\
As $E\in\mathscr{P}_{constr}(Y_\Sigma)$ the sets $A_{h,k}$ satisfy the following properties:
\begin{itemize}
\item the sets $\{A_{h,k} : k=1,\ldots,m\}$ are pairwise disjoint for every $h= 1,\ldots,m$\,; 
\item $\mathcal{H}^1\left(p(\partial^\ast E) \setminus \bigcup_{h<k}A_{h,k}\right) = 0$\,; 
\item $\Ha^1(A_{h,k} \cap \partial^\ast E^j) = 0$ for every $h,k\neq j$\,. 
\end{itemize}
Hence using \eqref{split} and the previous properties we have
\begin{eqnarray*}
P(E) & = & \sum_{h=1}^m P(E^h, \bigcup_{k=1}^m A_{h,k}) 
= \sum_{h,k=1}^m P(E^h,A_{h,k})\\
& = & \sum_{h<k} P(E^h,A_{h,k}) + \sum_{h>k} P(E^h,A_{h,k})\\
&=& 2\Ha^1(p(\partial^\ast E))
\end{eqnarray*}
as we wanted to prove.\\
If $A_{h,k} \cap \Sigma \neq \emptyset$ for some $h,k$, then one can repeat the previous argument decomposing $p(\partial^\ast E)$ with the sets $A_{h,k}$ in $D$ and with the sets in $A_{h',k'}$ in $\Sigma \setminus S$ and then use $(\ref{split})$.
\end{proof}
\begin{rem}\label{similar}
With a similar proof it is possible to prove that given  a $\mu$--measurable set $A\subset \mathbb{R}^2$  and $E\in\mathscr{P}_{constr}(Y_\Sigma)$ we have
\begin{equation*}
P(E,p^{-1}(A)) =2 \Ha^1(A\cap p(\partial^\ast E))\,.
\end{equation*}
\end{rem}

\begin{lemma}[Non-constancy]\label{nonconst}
Let $A\subset \mathbb{R}^2$ be a nonempty, open set such that $p^{-1}(A\setminus S)$ is connected.
Then, for every $E \in \mathscr{P}_{constr}(Y_\Sigma)$
\begin{equation*}
\Ha^1 (A \cap p(\partial^\ast E)) > 0\,.
\end{equation*}
\end{lemma}
\begin{proof}
By contradiction there exists $E \in \mathscr{P}_{constr}(Y_\Sigma)$ 
such that $\Ha^1 (A \cap p(\partial^\ast E)) = 0$. 
Then by Remark~\ref{similar} we obtain $
P(E, p^{-1}(A\setminus S)) = 2\Ha^1(A \cap p(\partial^\ast E)) =0$.
By Lemma~\ref{represent} there holds
$P(E^j, (p^{-1}(A\setminus S))^j) \leq P(E, p^{-1}(A\setminus S))=0$,
for every $j\in\{1,\ldots, m\}$
and $P(E^{j'}, (p^{-1}(A\setminus S))^{j'}) \leq P(E, p^{-1}(A\setminus S))=0$,
for every $j'\in\{m+1,\ldots, 2m\}$.

Hence $D\chi_{E^j}=0$
in $(p^{-1}(A\setminus S))^j$, that implies that the function $\chi_{E^j}$ is constant in $(p^{-1}(A\setminus S))^j$ 
for every $j\in\{1,\ldots, m\}$
(similarly $\chi_{E^{j'}}$ is constant in $(p^{-1}(A\setminus S))^{j'}$).
Moreover writing the set $p^{-1}(A\setminus S)$ as 
\begin{equation*}
p^{-1}(A\setminus S)=\bigcup_{j=1}^m\big[p^{-1}(A\setminus S)\cap \tilde{\pi}(D,j)\big]
\bigcup_{j'=m+1}^{2m}\big[p^{-1}(A\setminus S)\cap \tilde{\pi}(D',j')\big]
\end{equation*}
and applying $p$ to both sides we obtain that
\begin{equation*}
A\setminus S=\bigcup_{j=1}^m(p^{-1}(A\setminus S))^j
\bigcup_{j'=m+1}^{2m}(p^{-1}(A\setminus S))^{j'}.
\end{equation*}
Hence $A\setminus S$
is a connected set, union of open sets in which $\chi_{E^{j}}$ and $\chi_{E^{j'}}$ are constant. 
This implies that the value of all $\chi_{E^{j}}$ and  
$\chi_{E^{j'}}$ is the same for every $j,j'$ and, as a consequence, $\chi_E$ is constant in $p^{-1}(A\setminus S)$. This contradicts the validity of the constraint $i)$ of Definition~\ref{zerounoconstr}.
\end{proof}
The previous lemma implies that the projection via $p$ 
of the essential boundary of $E\in\mathscr{P}_{constr}(Y_\Sigma)$ touches $S$.
\begin{cor}\label{esse}
Let 
$E \in \mathscr{P}_{constr}(Y_\Sigma)$. Then
\begin{equation*}
S \subset \overline{p(\partial^\ast E)} \,.
\end{equation*}
\end{cor}
\begin{proof}
Suppose by contradiction that there exists $p_i \in S\setminus  \overline{p(\partial^\ast E)}$. 
Then there exists a ball $B$ with center in $p_i$ and such that 
$B\cap \overline{p(\partial^\ast E)} = \emptyset$. 
Notice that $p^{-1}(B\setminus \{p_i\})$ is path connected in $Y_\Sigma$. Indeed given two points $q_1$ and $q_2$ in $p^{-1}(B\setminus \{p_i\})$ it is possible to construct a connected path in $Y_\Sigma$ joining $q_1$ and $q_2$ following the identifications given by $\sim$ and crossing the cuts in the right order (see Definition~\ref{admissible} and the identifications defined in \eqref{ide2}); so we can apply Lemma $\ref{nonconst}$ to deduce that
\begin{displaymath}
\Ha^1(B\cap p(\partial^\ast E)) > 0\,,
\end{displaymath}
that is a contradiction.
\end{proof}

\begin{prop}\label{connected}
Consider $E\in  \mathscr{P}_{constr}(Y_\Sigma)$ such that $\Ha^1(\overline{p(\partial^\ast E)})<+\infty$. If at least one point of $S$
is contained in a connected component $\mathcal{C}$ of $\overline{p(\partial^\ast E)}$,
then the whole $S$ is contained in $\mathcal{C}$.

%
\end{prop}

\begin{proof}
For any $p_j\in S$ let $\mathcal{C}_j$ be the connected component of $\overline{p(\partial^\ast E)}$ containing $p_j$. Suppose by contradiction that there exists $k_1\neq k_2$ such that $\mathcal{C}_{k_1} \neq \mathcal{C}_{k_2}$. 
As a consequence $\overline{p(\partial^\ast E)}$ is not connected. Hence there exist two non-empty disjoint sets $A,B \subset \overline{p(\partial^\ast E)}$, relatively closed in $\overline{p(\partial^\ast E)}$ with $A\cup B = \overline{p(\partial^\ast E)}$. 
Moreover it is possible to choose $A$ and $B$ satifying the properties above and such that there exists $p_A, p_B \in S$ with $p_A \in A$ and $p_B \in B$. 
Notice that, thanks to condition $ii)$ in Definition \ref{zerounoconstr}, the set $\overline{p(\partial^\ast E)}$ is bounded. Hence $A$ and $B$ are compact in $\R^2$. Let $\varepsilon > 0$ be such that
\begin{displaymath}
A_\varepsilon \cap B = \emptyset\, ,
\end{displaymath}
where we have denoted by $A_\varepsilon$ the open $\varepsilon$-neighbourhood of $A$.
Notice that $\partial A_\varepsilon$ is a Lipschitz manifold for $\varepsilon$ small enough \cite{tubular}. Hence $\partial A_\varepsilon$ has a finite number of connected components and in particular it is a finite union of simple loops (see for example~\cite[Corollary 1]{connectper}) that we denote by $\{\gamma_i\}_{i=1,\ldots,q}$. As $p_A \in A_\varepsilon$,

\begin{equation}\label{indicescurve}
1 = \mbox{Ind}(\gamma_1 \circ \gamma_2 \ldots \circ \gamma_q  ,p_A) = \sum_{i=1}^q \mbox{Ind}(\gamma_i, p_A)\,,
\end{equation}
where we have denoted by $\mbox{Ind}(\gamma,p)$ the index of the loop $\gamma$ with respect to the point $p$ and by $\gamma \circ \sigma$ the concatenation of two curves $\gamma$ and $\sigma$. From \eqref{indicescurve} we infer that there exists a loop $\gamma_Q$ such that $\mbox{Ind}(\gamma_Q, p_A) = 1$ and, as $p_B$ belongs to the unbounded connected component of $\R^2$ described by $\gamma_Q$, we have also that $\mbox{Ind}(\gamma_Q, p_B) = 0$. 

Hence the loop $\gamma_Q$ is such that $\gamma_Q \cap A = \emptyset$, $\gamma_Q \cap B = \emptyset$ and it has index one with respect to at most $m-1$
points of $S$ and index zero with respect to at least one point of $S$. 
This implies that $\gamma_Q$ is crossing the cuts $\Sigma$ and $\Sigma'$ at least once and therefore  $p^{-1}(\gamma_Q)$ is a closed loop in the covering space $Y_\Sigma$. Indeed tracking the path $p^{-1}(\gamma_Q)$ in $Y_\Sigma$ one can see that for every crossing of the cuts $\Sigma$ and $\Sigma'$ in the covering space the path is continuing to the next sheet (following the identification given in \eqref{ide}). In doing so, $p^{-1}(\gamma_Q)$ is visiting all the sheets of $Y_\Sigma$ closing back at the starting point.

Additionally there exists an open $\varepsilon$-tubular neighborhood of $\gamma_Q$ such that $(\gamma_Q)_\varepsilon \cap A = \emptyset$ and $(\gamma_Q)_\varepsilon \cap B = \emptyset$.
We infer that $p^{-1}((\gamma_Q)_\varepsilon)$ is path connected. Indeed given two points $q_1,q_2 \in p^{-1}((\gamma_Q)_\varepsilon)$ it is enough to connect them to $p^{-1}(\gamma_Q)$ with a path that do not intersect the cuts and use that $p^{-1}(\gamma_Q)$ is a closed loop in $Y_\Sigma$.
This is a contradiction with Lemma \ref{nonconst}.

%
%
%
\end{proof}

The next theorem is a regularity result for $p(\partial^\ast E)$
when $E \in \mathscr{P}_{constr}(Y_\Sigma)$ is a minimizer of
Problem~\eqref{minpro2}.
To prove this theorem our strategy is to establish \emph{locally} an
equivalence between
Problem~\eqref{minpro2} and the partition problem \cite{ab1} and then use
the known regularity results for it (see, for
instance,~\cite{morel, tamanini}).

\begin{dfnz}[Local minimizer]
Given $A$ an open subset of $Y_\Sigma$,
we say that $E\in \mathscr{P}_{constr}(Y_\Sigma)$ is a local minimizer
for Problem~\eqref{minpro2}
in $A$ if for every $F\in \mathscr{P}_{constr}(Y_\Sigma)$ such that
$E\Delta F \subset \subset A$,
there holds
\begin{equation*}
P(E,A) \leq  P(F,A)\,.
\end{equation*}
\end{dfnz}

Consider the $m$--regular simplex in $\mathbb{R}^{m+1}$ centred in the
origin and call
$\{\alpha_1\,\ldots \,\alpha_m\}$
the vertices of the simplex and call 
$BV(A, \{\alpha_1\,\ldots \,\alpha_m\})$
the space of $BV$--functions with values in 
$\{\alpha_1\,\ldots \,\alpha_m\}$.

\begin{dfnz}[Local minimizer for the partition problem]\label{minpart}
Given $A$ an open bounded subset of $\R^2$, we say that
$u\in BV(A, \{\alpha_1\,\ldots \,\alpha_m\})$ is a local minimizer for the partition problem
in $A$ if for every $w\in BV(A, \{\alpha_1\,\ldots \,\alpha_m\})$
such that $\{u\neq w\} \subset \subset A$,
there holds
$$
\vert Du \vert \left(A \right)\leq  \vert Dw \vert \left(A \right)\,.
$$
\end{dfnz}

Given a set  $E\in \mathscr{P}_{constr}(Y_\Sigma)$,
consider $x\in \overline{p(\partial^\ast E)}$ such that $x\notin \Sigma \cup
\Sigma'$ and $r>0$ small enough such that $B_r(x) \cap (\Sigma \cup \Sigma') = \emptyset$.
The associated vector valued function
$u^\alpha$ in $BV(B_r(x), \{\alpha_1\,\ldots \,\alpha_m\})$
is canonically defined as
\begin{equation*}
u^{\alpha}(x) \coloneqq \alpha_j \quad\text{if}\; x\in E^j\cap B_r(x)\,,
\end{equation*}
for $j=1\,,\ldots\,,m$.
Notice that by construction there holds
\begin{equation}\label{jp}
p(\partial^\ast E) \cap B_r(x) =J_{u^\alpha}\,,
\end{equation}
where $J_{u^\alpha}$ is the jump set of $u^\alpha$.

\begin{lemma}\label{equivalenzalocale}
Suppose that $E_{\min}\in \mathscr{P}_{constr}(Y_\Sigma)$
is a local minimizer
for Problem~\eqref{minpro2} in  $p^{-1}(B_r(x))$
with $x\in \overline{p(\partial^\ast E)}$ and $r>0$ 
such that $B_r(x) \cap (\Sigma \cup \Sigma') = \emptyset$.
Then the associated vector valued function $u_{\min}^\alpha$
in $BV(B_r(x), \{\alpha_1\,\ldots \,\alpha_m\})$ is a local minimizer
for the partition problem in $p^{-1}(B_r(x))$.
\end{lemma}
\begin{proof}
Consider  any  $w\in BV(B_r(x),\{\alpha_1\,\ldots \,\alpha_m\})$
such that $\{u_{\min}^\alpha \neq w\} \subset \subset B_r(x)$.
We associate to $w$ a set $F$ in  $\mathscr{P}_{constr}(Y_\Sigma)$
defining its characteristic function as
\begin{equation*}
\chi_ {F^{j}}(x) \coloneqq\left\{
\begin{array}{ll}
1 &\text{if}\; w(x)=\alpha_j\,,\\
0 &\text{otherwise}\,
\end{array}
\right.
\end{equation*}
for $j=1\,,\ldots\,,m$ (see Remark~\ref{presc}).
By construction we have that $J_{w}=p(\partial^\ast F) \cap B_r(x)$.

Then
applying Remark~\ref{similar} and~\eqref{jp} we obtain
\begin{align*}
&2\vert Du^\alpha_{\min}\vert
(B_r(x))=2\mathcal{H}^1\left(J_{u^\alpha_{\min}}\right)\\
&= 2\mathcal{H}^1\left(p(\partial^\ast E_{\min})\cap B_r(x)
\right)=P(E_{\min},p^{-1}(B_r(x)))\\
&\leq P(F,p^{-1}(B_r(x)))=2\mathcal{H}^1\left(p(\partial^\ast F) \cap
B_r(x)\right)\\
& =2\mathcal{H}^1\left( J_w \right)=2\vert Dw \vert(B_r(x))\,,
\end{align*}
hence
\begin{equation*}
\vert Du^\alpha_{\min}\vert (B_r(x))\leq \vert Dw \vert(B_r(x))\,.
\end{equation*}
\end{proof}

\begin{rem}
We have exhibited a way
to pass \emph{locally} from
our minimization problem in the covering space to a problem of minimal
partition in $\mathbb{R}^2$
for~\eqref{minpro2}.
This is enough for our aim, that is obtaining the regularity of the
minimizer $E$,
but clearly it is possible, with a similar procedure, to show that given
a local minimizer for the minimal partition problem, then the associated
set $E\in  \mathscr{P}_{constr}(Y_\Sigma)$  is a local minimizer
for~\eqref{minpro2}.
We underline that in general  the equivalence between
the partition problem and Problem~\eqref{minpro2} does not hold
\emph{globally}.
\end{rem}

\begin{teo}[Regularity]\label{regularity}
Given $E_{min} \in \mathscr{P}_{constr}(Y_\Sigma)$
a minimizer of Problem~\eqref{minpro2}, then
there holds
\begin{equation}\label{reg1}
\mathcal{H}^1(\overline{p(\partial^\ast E_{\min})} \setminus
p(\partial^\ast E_{\min}))=0\,.
\end{equation}
Moreover $\overline{p(\partial^\ast E_{\min})}$ is a finite union of
segments meeting at triple junctions with angles of $120$ degrees.
\end{teo}
\begin{proof}
Consider $x\in \overline{p(\partial^\ast E)}$ and suppose without loss of
generality
(thanks to Remark~\ref{independence}) that $x\notin \Sigma \cup \Sigma'$;
then there exists $r>0$ such that $B_r(x)\cap (\Sigma \cup
\Sigma')=\emptyset$ and
$E_{\min}\in \mathscr{P}_{constr}(Y_\Sigma)$
is a local minimizer
 for~\eqref{minpro2} in  $p^{-1}(B_r(x))$.
 Then by Lemma~\ref{equivalenzalocale} the associated function
 $u_{\min}^\alpha$ is a local minimizer for the partition problem.
Thanks to the regularity results for the local partition problem
(see~\cite[Theorem $4.7$]{tamanini})
we have that
$\mathcal{H}^1(\overline{p(\partial^\ast E_{\min})} \setminus
p(\partial^\ast E_{\min}))=0$.
Moreover $\overline{p(\partial^\ast E_{\min})}$ inherits all the
regularity properties of the minimum
of the partition problem, namely
the set $\overline{p(\partial^\ast E_{\min})}$ is finite union of segments
meeting  at triple junctions with angles of $120$ degree.
\end{proof}

\subsection{Proof of the equivalence}
In this section we 
prove that the minimization problem~\eqref{minpro2}
is equivalent to the Steiner problem in $\R^2$. 

First of all we need to prove that, given a solution of the Steiner problem for $S$, 
we can find a set $E\in \mathscr{P}_{constr}(Y_\Sigma)$ such that $\overline{p(\partial^\ast E)}$
is the Steiner network. We prove this statement for
a smaller class of network, namely for the connected networks without loops. This result will be used again in Section~\ref{famiglie}.

\begin{dfnz}\label{conetw}
A connected network is a finite union of $C^0$ injective curves
that intersect each other only at their end points.
We say that a connected network $\mathscr{S}$ connects the points of $S$ if $S\subset \mathscr{S}$ and 
the end points of the curves of $\mathscr{S}$ either have order one and are points of $\mathscr{S}$ or have order greater or equal than one.
In the first case the end points are called leaves, in the latter case they are called multipoint.
We call $L\subset \mathscr{S}$ the set of all leaves of $\mathscr{S}$
\end{dfnz}

\begin{prop}\label{costruzione}
Consider $S=\{p_1, \ldots, p_m\}$ and $\mathscr{S}$
a connected network without loops that connects the $m$ points of $S$.
Then, for an appropriate relabeling of the points of $S$ 
there exists 
an admissible pair of cuts ${\hat{\bf{{\Sigma}}}} \in \mbox{\textbf{Cuts}}(S)$ and
a set $E_{\mathscr{S}} \in \mathscr{P}_{constr}(Y_{\hat{\Sigma}})$
such that $\overline{p(\partial^\ast E_{\mathscr{S}})}=\mathscr{S}$.
\end{prop}
\begin{proof}
First of all we prove that, up to a permutation of the labelling of the point of $S$, 
there exists an admissible pair of cuts $\hat{{\bf{\Sigma}}} \in \mbox{\textbf{Cuts}}(S)$ 
such that $\hat{{\bf{\Sigma}}}\cap \mathscr{S}=S$. 

We notice that in order to prove the previous claim it is sufficient to find $\hat{\Sigma} \in Cuts(S)$ 
such that $\hat{\Sigma} \cap \mathscr{S}=S$.
Then by a continuous deformation it is immediate to construct 
$\hat{{\bf{\Sigma}}} \in \mbox{\textbf{Cuts}}(S)$ with $\hat{{\bf{\Sigma}}}\cap \mathscr{S}=S$. 
We build $\hat\Sigma$ in a constructive way. 
We remind that from Definition~\ref{conetw} follows that the set $S$ can be written as
$L\cup \mathcal{M}:=\{\ell_1,\ldots,\ell_h\}\cup\{m_1,\ldots,m_k\}$ where $L$ is the set of leaves
and $\mathcal{M}$ is a subset (possibly empty) of the set of all the multipoints of $\mathscr{S}$ and 
$m=h+k$.
The first step of our construction is the following:
fix $\ell_1\in L$ and follow the network $\mathscr{S}$ with the rule that at every multipoint 
we proceed along the closest curve with respect to the clockwise rotation. 
As the network is without loops, this procedure ends in a leaf $\ell_2\neq\ell_1$.
We call $N_{\ell_1}$ the subnetwork described by the just defined procedure, 
$N_{\ell_1}$ contains $\ell_1,\ell_2$ and possibly some points of $\mathcal{M}$, let us say
$m_i$ with $i\in\{1,\ldots,j\}, j\leq k$.
Then there exists a Lipschitz curve
$\Sigma_1$  that connects $\ell_1$ to $m_1$, for $i\in\{1,\ldots,j-1\}$ there exist
Lipschitz curves $\Sigma_{i+1}$ that connects $m_i$ to $m_{i+1}$ 
and  $\Sigma_{j+1}$ a Lipschitz curve from $m_j$ to $\ell_2$. Moreover
we can choose all the Lipschitz curves in such a way that 
they do not intersect $\mathscr{S}\setminus S$ 
(for instance they can be obtained by  
continuous deformation of the subnetwork $N_{\ell_1}$).
Step $2$ to step $h$ of the procedure are nothing else than an 
iteration of the procedure of the first step, starting from $\ell_i$ with $i\in\{2,\ldots,h-1\}$ with the extra requirement that 
at the $n-th$ step one ones does not connect with Lipschitz curves the points $m_i$, already visited in the steps $1$ to $n-1$ (see Figure \ref{puntiinterni}). 
This will produce a family of Lipschitz curves $\{\Sigma_i\}_{i=1,\ldots,m-1}$
connecting the points of $S$ and not intersecting $\mathscr{S}\setminus S$. 
Then $\hat{\Sigma}=\cup_{i=1}^{m-1}\Sigma_i$ is the desired set of cuts.

Now we describe how to associate to $\mathscr{S}$ a set $E_{\mathscr{S}}$ 
in the covering space $Y_{\hat{\Sigma}}$.
For $j=1,\ldots,m-1$ the set $E_{\mathscr{S}}^{m+1-j}$ is defined as the open set
such that its boundary is composed by $\Sigma_j$ and the part of $\mathscr{S}$ 
connecting $\tilde{p}_j$ and $\tilde{p}_{j+1}$
and $E_{\mathscr{S}}^1=\mathbb{R}^2\setminus \cup_{j=1}^{m-1}E_{\mathscr{S}}^{m+1-j}$.
Thanks to Remark~\ref{presc} the set $E_{\mathscr{S}}$ is well defined.
By construction it is trivial that $E_{\mathscr{S}}$ 
satisfies the constraints of Definition~\ref{zerounoconstr} 
and that $\overline{p(\partial^\ast E_{\mathscr{S}})}=\mathscr{S}$.
\end{proof}


\begin{rem}
The choice of $E^j$ in the previous construction is not arbitrary: if one chooses differently the sets $E^j$,
one obtains a different set
with perimeter greater than the perimeter of $E_{\mathscr{S}}$. 
\end{rem}

\begin{figure}[H]
\centering
\begin{tikzpicture}[scale=1.3]
\path[font=\footnotesize]
(1,1) node[above]{$p_3=l_3$}
(1,-1) node[below]{$p_2=l_2$}
(-1,-1) node[below]{$p_1=l_1$}
(-1,1) node[above]{$p_4=l_4$};
\fill[black](1,1) circle (1.7pt);    
\fill[black](1,-1) circle (1.7pt);    
\fill[black](-1,1) circle (1.7pt);    
\fill[black](-1,-1) circle (1.7pt);    
\draw[rotate=90]
(1,-1)--(0.42,0)
(1,1)--(0.42,0)
(0.42,0)--(-0.42,0)
(-0.42,0)--(-1,-1)
(-0.42,0)--(-1,1);
\draw[rotate=90, dashed]
(-1,-1)to[out=-150,in=-60, looseness=1] (1,-1)
(-1,1)to[out=160,in=-160, looseness=1] (-1,-1)
(1,1)to[out=-20,in=20, looseness=1] (1,-1);
\draw[green!50!black, thick, dashed, rotate=90]
(1.05,-1)--(0.47,0)
(1.05,1)--(0.47,0);
\draw[red, thick, dashed, rotate=90]
(-0.47,0)--(-1.05,-1)
(-0.47,0)--(-1.05,1);
\draw[blue!20!black, thick, dashed, rotate=90]
(-0.37,-0.1)--(0.37,-0.1)
(-0.37,-0.1)--(-0.95,-1)
(0.37,-0.1)--(0.95,-1);
\draw[thick, blue!20!black, shift={(0.65,0.72)}, scale=1, rotate=-60]
(0,0)to[out= -45,in=135, looseness=1] (0.1,-0.1)
(0,0)to[out= -135,in=45, looseness=1] (-0.1,-0.1);
\draw[thick, green!50!black, shift={(0.55,0.79)}, scale=1, rotate=-240]
(0,0)to[out= -45,in=135, looseness=1] (0.1,-0.1)
(0,0)to[out= -135,in=45, looseness=1] (-0.1,-0.1);
\draw[thick, blue!20!black, shift={(0.65,-0.72)}, scale=1, rotate=60]
(0,0)to[out= -45,in=135, looseness=1] (0.1,-0.1)
(0,0)to[out= -135,in=45, looseness=1] (-0.1,-0.1);
\draw[thick, red, shift={(-0.55,-0.79)}, scale=1, rotate=-60]
(0,0)to[out= -45,in=135, looseness=1] (0.1,-0.1)
(0,0)to[out= -135,in=45, looseness=1] (-0.1,-0.1);
\draw[thick, red, shift={(0.55,-0.79)}, scale=1, rotate=-120]
(0,0)to[out= -45,in=135, looseness=1] (0.1,-0.1)
(0,0)to[out= -135,in=45, looseness=1] (-0.1,-0.1);
\draw[thick, green!50!black, shift={(-0.55,0.79)}, scale=1, rotate=60]
(0,0)to[out= -45,in=135, looseness=1] (0.1,-0.1)
(0,0)to[out= -135,in=45, looseness=1] (-0.1,-0.1);
\draw[thick, blue!20!black, shift={(0.1,0)}, scale=1, rotate=0]
(0,0)to[out= -45,in=135, looseness=1] (0.1,-0.1)
(0,0)to[out= -135,in=45, looseness=1] (-0.1,-0.1);
\end{tikzpicture}\qquad
\begin{tikzpicture}[scale=1.3]
\path[font=\footnotesize]
(0,0.42)node[left]{$p_3=m_1$}
(1,1) node[above]{$p_4=l_3$}
(1,-1) node[below]{$p_2=l_2$}
(-1,-1) node[below]{$p_1=l_1$}
(-1,1) node[above]{$p_5=l_4$};
\fill[black](0,0.42)circle (1.7pt);   
\fill[black](1,1) circle (1.7pt);    
\fill[black](1,-1) circle (1.7pt);    
\fill[black](-1,1) circle (1.7pt);    
\fill[black](-1,-1) circle (1.7pt);    
\draw[rotate=90]
(1,-1)--(0.42,0)
(1,1)--(0.42,0)
(0.42,0)--(-0.42,0)
(-0.42,0)--(-1,-1)
(-0.42,0)--(-1,1);
\draw[rotate=90, dashed]
(-1,-1)to[out=-150,in=-120, looseness=1] (0.42,0)
(0.42,0)to[out=-120,in=-100, looseness=1] (1,-1)
(-1,1)to[out=160,in=-160, looseness=1] (-1,-1)
(1,1)to[out=-20,in=20, looseness=1] (1,-1);
\draw[green!50!black, thick, dashed, rotate=90]
(1.05,-1)--(0.47,0)
(1.05,1)--(0.47,0);
\draw[red, thick, dashed, rotate=90]
(-0.47,0)--(-1.05,-1)
(-0.47,0)--(-1.05,1);
\draw[blue!20!black, thick, dashed, rotate=90]
(-0.37,-0.1)--(0.37,-0.1)
(-0.37,-0.1)--(-0.95,-1)
(0.37,-0.1)--(0.95,-1);
\draw[thick, blue!20!black, shift={(0.65,0.72)}, scale=1, rotate=-60]
(0,0)to[out= -45,in=135, looseness=1] (0.1,-0.1)
(0,0)to[out= -135,in=45, looseness=1] (-0.1,-0.1);
\draw[thick, green!50!black, shift={(0.55,0.79)}, scale=1, rotate=-240]
(0,0)to[out= -45,in=135, looseness=1] (0.1,-0.1)
(0,0)to[out= -135,in=45, looseness=1] (-0.1,-0.1);
\draw[thick, blue!20!black, shift={(0.65,-0.72)}, scale=1, rotate=60]
(0,0)to[out= -45,in=135, looseness=1] (0.1,-0.1)
(0,0)to[out= -135,in=45, looseness=1] (-0.1,-0.1);
\draw[thick, red, shift={(-0.55,-0.79)}, scale=1, rotate=-60]
(0,0)to[out= -45,in=135, looseness=1] (0.1,-0.1)
(0,0)to[out= -135,in=45, looseness=1] (-0.1,-0.1);
\draw[thick, red, shift={(0.55,-0.79)}, scale=1, rotate=-120]
(0,0)to[out= -45,in=135, looseness=1] (0.1,-0.1)
(0,0)to[out= -135,in=45, looseness=1] (-0.1,-0.1);
\draw[thick, green!50!black, shift={(-0.55,0.79)}, scale=1, rotate=60]
(0,0)to[out= -45,in=135, looseness=1] (0.1,-0.1)
(0,0)to[out= -135,in=45, looseness=1] (-0.1,-0.1);
\draw[thick, blue!20!black, shift={(0.1,0)}, scale=1, rotate=0]
(0,0)to[out= -45,in=135, looseness=1] (0.1,-0.1)
(0,0)to[out= -135,in=45, looseness=1] (-0.1,-0.1);
\end{tikzpicture}\qquad
\begin{tikzpicture}[scale=1.3]
\path[font=\footnotesize]
(0,-0.42)node[left]{$p_3=m_1$}
(0,0.42)node[left]{$p_4=m_2$}
(1,1) node[above]{$p_5=l_3$}
(1,-1) node[below]{$p_2=l_2$}
(-1,-1) node[below]{$p_1=l_1$}
(-1,1) node[above]{$p_6=l_4$};
\fill[black](0,0.42)circle (1.7pt);
\fill[black](0,-0.42)circle (1.7pt);      
\fill[black](1,1) circle (1.7pt);    
\fill[black](1,-1) circle (1.7pt);    
\fill[black](-1,1) circle (1.7pt);    
\fill[black](-1,-1) circle (1.7pt);    
\draw[rotate=90]
(1,-1)--(0.42,0)
(1,1)--(0.42,0)
(0.42,0)--(-0.42,0)
(-0.42,0)--(-1,-1)
(-0.42,0)--(-1,1);
\draw[rotate=90, dashed]
(-1,1)to[out=160,in=-160, looseness=1] (-0.42,0)
(-0.42,0)to[out=160,in=-160, looseness=1] (-1,-1)
(-1,-1)to[out=-150,in=-120, looseness=1] (0.42,0)
(0.42,0)to[out=-120,in=-100, looseness=1] (1,-1)
(1,1)to[out=-20,in=20, looseness=1] (1,-1);
\draw[green!50!black, thick, dashed, rotate=90]
(1.05,-1)--(0.47,0)
(1.05,1)--(0.47,0);
\draw[red, thick, dashed, rotate=90]
(-0.47,0)--(-1.05,-1)
(-0.47,0)--(-1.05,1);
\draw[blue!20!black, thick, dashed, rotate=90]
(-0.37,-0.1)--(0.37,-0.1)
(-0.37,-0.1)--(-0.95,-1)
(0.37,-0.1)--(0.95,-1);
\draw[thick, blue!20!black, shift={(0.65,0.72)}, scale=1, rotate=-60]
(0,0)to[out= -45,in=135, looseness=1] (0.1,-0.1)
(0,0)to[out= -135,in=45, looseness=1] (-0.1,-0.1);
\draw[thick, green!50!black, shift={(0.55,0.79)}, scale=1, rotate=-240]
(0,0)to[out= -45,in=135, looseness=1] (0.1,-0.1)
(0,0)to[out= -135,in=45, looseness=1] (-0.1,-0.1);
\draw[thick, blue!20!black, shift={(0.65,-0.72)}, scale=1, rotate=60]
(0,0)to[out= -45,in=135, looseness=1] (0.1,-0.1)
(0,0)to[out= -135,in=45, looseness=1] (-0.1,-0.1);
\draw[thick, red, shift={(-0.55,-0.79)}, scale=1, rotate=-60]
(0,0)to[out= -45,in=135, looseness=1] (0.1,-0.1)
(0,0)to[out= -135,in=45, looseness=1] (-0.1,-0.1);
\draw[thick, red, shift={(0.55,-0.79)}, scale=1, rotate=-120]
(0,0)to[out= -45,in=135, looseness=1] (0.1,-0.1)
(0,0)to[out= -135,in=45, looseness=1] (-0.1,-0.1);
\draw[thick, green!50!black, shift={(-0.55,0.79)}, scale=1, rotate=60]
(0,0)to[out= -45,in=135, looseness=1] (0.1,-0.1)
(0,0)to[out= -135,in=45, looseness=1] (-0.1,-0.1);
\draw[thick, blue!20!black, shift={(0.1,0)}, scale=1, rotate=0]
(0,0)to[out= -45,in=135, looseness=1] (0.1,-0.1)
(0,0)to[out= -135,in=45, looseness=1] (-0.1,-0.1);
\end{tikzpicture}
\caption{Examples of three possible situations occouring in Proposition \ref{costruzione}}\label{puntiinterni}
\end{figure}
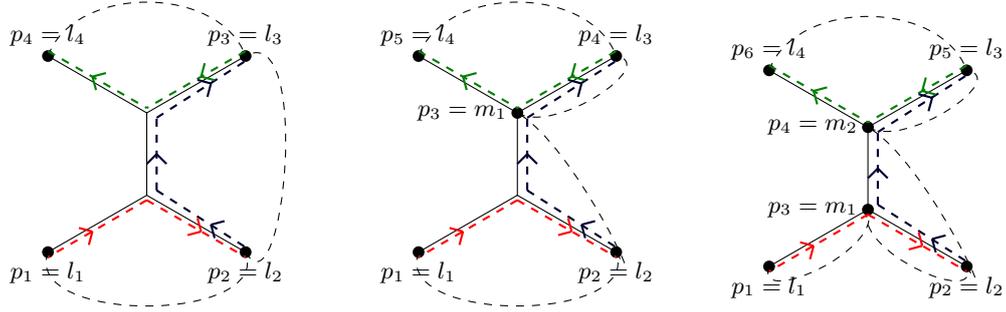

\begin{teo}\label{samelength}
The Steiner problem 
and Problem~\eqref{minpro2} are equivalent.
In particular if $E\in\mathscr{P}_{constr}(Y_\Sigma)$ is a minimizer of Problem~\eqref{minpro2},
then $\overline{p(\partial^\ast E})$ is a solution of the Steiner
problem and
if $\mathscr{S}$ is a minimizer for the Steiner problem,
then its associated set $E_\mathscr{S}$ (constructed as in Proposition~\ref{costruzione})
is a minimizer for Problem~\eqref{minpro2}.
\end{teo}
\begin{proof}
Consider $E$ a minimizer for Problem~\eqref{minpro2} in $Y_\Sigma$
and $\mathscr{S}$ a minimizer for the Steiner problem. 
Thanks to Proposition~\ref{costruzione} there exists
an admissible pair of cuts ${\hat{\bf{{\Sigma}}}} \in \mbox{\textbf{Cuts}}(S)$ and a set $E_{\mathscr{S}} \in \mathscr{P}_{constr}(Y_{\hat{\Sigma}})$
such that $\overline{p(\partial^\ast E_{\mathscr{S}})}=\mathscr{S}$.
Thanks to Remark~\ref{independence} Problem~\eqref{minpro2}
is independent on the choice of the cuts. Therefore
\begin{equation}\label{f}
P(E, Y_\Sigma) \leq P(E_\mathscr{S}, Y_{\hat{\Sigma}})\,.
\end{equation}
Using Proposition~\ref{dueh} 
and inequality~\eqref{f} we have
\begin{equation}\label{d}
\mathcal{H}^1(p(\partial^\ast E))=\frac 12 P(E, Y_{\Sigma}) \leq \frac 12 P(E_\mathscr{S}, Y_{\hat{\Sigma}})=
\mathcal{H}^1(p(\partial^\ast E_\mathscr{S})) \leq \mathcal{H}^1(\mathscr{S})\,.
\end{equation}
Thanks to Theorem~\ref{regularity} we know that
\begin{equation}\label{g}
\mathcal{H}^1(\overline{p(\partial^\ast E)})=\mathcal{H}^1(p(\partial^\ast E))\,.
\end{equation}
Thanks to Corollary~\ref{esse} we have $S\subset \overline{p(\partial^\ast E)}$, 
and by Proposition~\ref{connected}
there exists a connected component  $\mathcal{C}_{E}$ of $\overline{p(\partial^\ast E)}$ 
that contains $S$. 
Hence $\mathcal{C}_{E}$ is a competitor for the Steiner problem for $S$ in $\mathbb{R}^2$, 
therefore from the minimality of $\mathscr{S}$ we get
\begin{equation}\label{e}
\mathcal{H}^1(\mathscr{S})\leq\mathcal{H}^1(\mathcal{C}_{E})
\leq\mathcal{H}^1(\overline{p(\partial^\ast E)})\,.
\end{equation}
Combining~\eqref{d},~\eqref{g} and~\eqref{e} we have 
\begin{equation*}
\mathcal{H}^1(\overline{p(\partial^\ast E)})=\mathcal{H}^1(\mathscr{S})\,.
\end{equation*}
We obtain as well that 
$\mathcal{H}^1(\mathcal{C}_{E})=\mathcal{H}^1(\overline{p(\partial^\ast E)})$; 
therefore using again Theorem~\ref{regularity} we infer that $\overline{p(\partial^\ast E})$ is connected.
The set $\overline{p(\partial^\ast E)}$ is a connected set that joints the point of $S$
and (by the minimality of $\mathscr{S}$) such that, for every connected set $\mathscr{T}\subset \R^2$
that connects the point of $S$
\begin{equation*}
\mathcal{H}^1(\overline{p(\partial^\ast E)})= \mathcal{H}^1(\mathscr{S})\leq  \mathcal{H}^1(\mathscr{T}) \,.
\end{equation*}
Hence $\overline{p(\partial^\ast E)}$ is a minimizer for the Steiner problem.
On the other hand the constrained set $E_\mathscr{S}$  has the same perimeter of $E$,
hence is a solution of Problem~\eqref{minpro2}.
\end{proof}

\section{Calibrations}\label{seccal}
In this section we introduce the notion of calibration for the minimum problem \eqref{minpro2} and we show some explicit examples. In doing so, it is often convenient to consider vector fields that are not continuous and for which a divergence theorem still holds. For this reason we employ the notion of \emph{approximately regular vector field} (in a slightly stronger version than in~\cite{mumford}) and then we generalize it to the covering space setting.

\begin{dfnz}[Approximately regular vector fields on $\R^n$]
Given $A\subset \R^{n}$, a Borel vector field $\Phi: A \rightarrow \R^{n}$ is approximately regular
if it is bounded and for every Lipschitz hypersurface $\mathscr{M}$ in $\R^{n}$, $\Phi$ admits traces on $\mathscr{M}$ on the two sides of $\mathscr{M}$ (denoted by $\Phi^+$ and $\Phi^-$) and 
\begin{equation}\label{app}
\Phi^+(x) \cdot \nu_{\mathscr{M}}(x) = \Phi^-(x) \cdot \nu_{\mathscr{M}}(x) = \Phi(x) \cdot \nu_\mathscr{M}(x),
\end{equation}
for $\mathcal{H}^{n-1}$--a.e. $x \in \mathscr{M}\cap A$.
\end{dfnz}  

\begin{dfnz}[Approximately regular vector fields on $Y_\Sigma$]\label{appreg}
Given $\Phi: Y_\Sigma\rightarrow \R^2$, we say that it is \emph{approximately regular} in $Y_\Sigma$ if
$\Phi^j$
and
$\Phi^{j'}$
(see Definition~\ref{pullback})
 are \emph{approximately regular} for every $j=1,\ldots,m$ and $j'=m+1,\ldots,2m$.  
\end{dfnz}

\begin{dfnz}[Calibration on coverings]\label{caliconvering}
Given $E\in\mathscr{P}_{constr}(Y_\Sigma)$,
a calibration for $E$
 (with respect to the minimum problem \eqref{minpro2})
is an \emph{approximately regular} vector field $\Phi :Y_{\Sigma}\to\mathbb{R}^2$ such that:
\begin{enumerate}
\item [(\textbf{1})] $\div\Phi=0$ (in the sense of the distributions);
\item [(\textbf{2})] $\vert \Phi^i (x) - \Phi^j (x)\vert \leq 2$ 
for every $i,j = 1,\ldots m$ and for every $x\in D$;
\item  [(\textbf{3})] $\int_{Y_\Sigma} \Phi \cdot D\chi_E=P(E)$.
\end{enumerate}
\end{dfnz}

\begin{rem}\label{twoinsteadofone}
At first sight the size condition (\textbf{2}) may sounds different in comparison with the classical notion of paired calibration. This difference is only apparent. Indeed we choose to minimize $P(E)$ that thanks to Proposition \ref{dueh} is equal to the double of the length of the minimal network on to the base space.
\end{rem}

\begin{prop}[Divergence theorem on coverings]\label{divteo}
Consider $E,F\in \mathscr{P}_{constr}(Y_\Sigma)$ and
let $\Phi : Y_\Sigma \to \R^2$ be an approximately regular vector field such that $\div \Phi = 0$ 
(in the sense of the distributions)
in $Y_\Sigma$.
Then
\begin{equation}\label{booh}
\int_{Y_\Sigma} \Phi \cdot D\chi_E = \int_{Y_\Sigma} \Phi \cdot D\chi_F.
\end{equation}  
\end{prop}
\begin{proof}
See Appendix \ref{appe}.
\end{proof}
In the following theorem we prove that our notion of calibration is indeed meaningful,
 in the sense that the existence of a calibration for a given $E\in \mathscr{P}_{constr}(Y_\Sigma)$ 
 implies the minimality of $E$ for Problem~\eqref{minpro2}.

\begin{teo}\label{impli}
If $\Phi :Y_{\Sigma}\to\mathbb{R}^{2}$ is a calibration for $E$,
then $E$ is a minimizer of Problem~\eqref{minpro2}.
\end{teo}
\begin{proof}
Let $\Phi :Y_{\Sigma}\to\mathbb{R}^{2}$ be a  calibration for $E\in \mathscr{P}_{constr}(Y_\Sigma)$ 
and let $F\in \mathscr{P}_{constr}(Y_\Sigma)$ a competitor. 
By Proposition~\ref{divteo} and (\textbf{1})  of Definition~\ref{caliconvering} we have
\begin{equation}\label{a}
\int_{Y_\Sigma} \Phi \cdot D\chi_E = \int_{Y_\Sigma} \Phi \cdot D\chi_F\,
\end{equation}
and thanks to property (\textbf{3}) of Definition~\ref{caliconvering}
\begin{equation}\label{b}
P(E) = \int_{Y_\Sigma} \Phi \cdot D\chi_E\,.
\end{equation}
Moreover, using~\eqref{derivata} we have

\begin{align*}\label{c}
 \int_{Y_\Sigma} \Phi \cdot D\chi_F
&= \sum_{j=1}^m \int_{D} \Phi^j\cdot D\chi_{F^j}
+ \sum_{j'=m+1}^{2m} \int_{\Sigma} \Phi^{j'}\cdot D\chi_{F^{j'}}\\
&= \sum_{j=1}^m\int_{\partial^\ast F^j \cap D}\Phi^j \cdot \nu_{F^{j}} \, d\Ha^{1}
+ \sum_{j'=m+1}^{2m} \int_{\partial^\ast F^{j'} \cap \Sigma}\Phi^{j'} \cdot \nu_{F^{j'}} \, d\Ha^{1}\\
&= \int_{p(\partial^\ast F)\cap D}\sum_{j=1}^m \Phi^j \cdot \nu_{F^{j}}\chi_{\partial^*F^j} \, d\Ha^{1}
+  \int_{p(\partial^\ast F) \cap \Sigma} \sum_{j'=m+1}^{2m}\Phi^{j'} \cdot \nu_{F^{j'}} \chi_{\partial^*F^{j'}}\, d\Ha^{1}\\
&\leq  \int_{p(\partial^\ast F)\cap D}
\Big|
\sum_{j=1}^m \Phi^j \cdot \nu_{F^{j}}\chi_{\partial^*F^j} \Big|\, d\Ha^{1}
+  \int_{p(\partial^\ast F) \cap \Sigma}\Big| 
\sum_{j'=m+1}^{2m}\Phi^{j'} \cdot \nu_{F^{j'}}\chi_{\partial^*F^{j'}}\Big| \, d\Ha^{1}.
\end{align*}
As $F\in \mathscr{P}_{constr}(Y_\Sigma)$, for $\Ha^1$--a.e.  $x\in p(\partial^\ast F) \cap D$ 
there exist exactly two distinct indices $j_1,j_2\in \{1,\ldots, m\}$ such that 
$x\in \partial^\ast F^{j_1}\cap \partial^\ast F^{j_2}$ and $\nu_{F^{j_1}}=-\nu_{F^{j_2}}$. 
Therefore using condition (\textbf{2})
 of Definition~\ref{caliconvering} and the usual identifications given by $\sim$ we get that
\begin{equation}\label{c}
\int_{Y_\Sigma} \Phi \cdot D\chi_F \leq 2\Ha^1(p(\partial^\ast F)) = P(F),
\end{equation}
where the last equality follows from Proposition~\ref{dueh}.\\
Combining Equations~\eqref{a},~\eqref{b} and~\eqref{c} one obtains 
\begin{equation*}
P(E) = \int_{Y_\Sigma} \Phi \cdot D\chi_E= \int_{Y_\Sigma} \Phi \cdot D\chi_F\leq P(F)\,.
\end{equation*}
\end{proof}

\begin{rem}\label{const}
Given $\Phi: Y_\Sigma \rightarrow \R^2$ a calibration for $E\in \mathscr{P}_{constr}(Y_\Sigma)$, 
then  for every $c\in \R^2$ we have that
$\Phi + c$ is a calibration for $E$.
Indeed if $\Phi$ is a calibration for $E\in \mathscr{P}_{constr}(Y_\Sigma)$ 
then it is easy to see that properties (\textbf{1}) and  (\textbf{2}) 
hold for $\Phi + c$ as well. 
It remains to show that if $\int_{Y_\Sigma} \Phi \cdot D\chi_{E} = P(E)$, then
\begin{displaymath}
\int_{Y_\Sigma} (\Phi + c) \cdot D\chi_{E} = P(E)\,,
\end{displaymath} 
that is that for every $c\in \R^2$ we have $\int_{Y_\Sigma} c \cdot D\chi_E = 0$. 
Following the computation in the proof of Theorem~\ref{divteo} we have that 
\begin{equation*}
\int_{Y_\Sigma} c \cdot D\chi_E = \int_{\partial \Omega} c \cdot \nu_{\partial \Omega}\, d\Ha^1 = 0\,.
\end{equation*}
\end{rem}

\subsection{Examples of calibrations}\label{excal}
We present here several examples of calibrations for Steiner configurations in the covering space setting. In the figures below the vector field is implicitly defined as the constant $\Phi=(0,0)$, where no arrows are drawn. Notice that, thanks to Definition \ref{appreg}, a calibration can admit discontinuities in the domain of definition provided that \eqref{app} is fulfilled.

\begin{ex}[Calibration for the segment]

\end{ex}
In order to introduce the reader to the calibration method in our setting we start with the trivial example of the minimality of the segment. In particular we show
that the set $E\in\mathscr{P}_{constr}(Y_\Sigma)$
defined in such a way that the 
closure of its essential boundary is the segment connecting $p_1$ and
$p_2$ is the minimizer of Problem \eqref{minpro2}.
We recall that in this case the number of sheets of the covering is two.
We define $E^1$ (resp. $E^2$) as the coloured subset of $(D,1)$ (resp. $(D,2)$) in Figure~\ref{minduefig} and the set $E$ is obtained as explained in Remark~\ref{presc}.   
\begin{figure}[H]
\centering
\begin{tikzpicture}
\node[inner sep=0pt] at (0,0)
    {\includegraphics[width=0.95\textwidth]{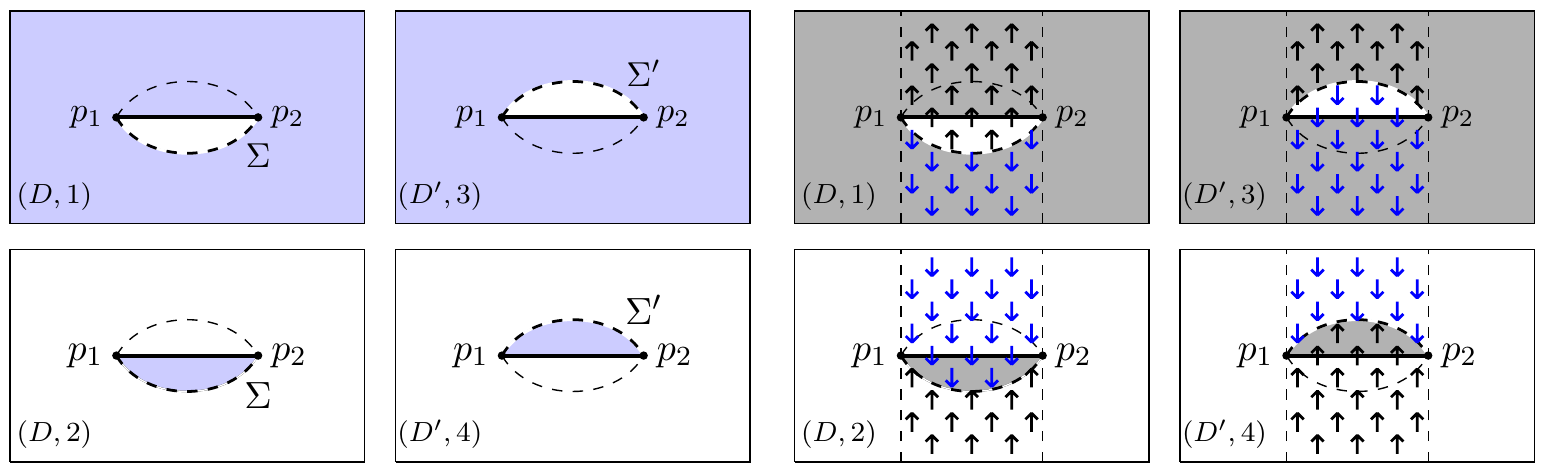}};
\end{tikzpicture}
\caption{The candidate minimizer $E$ and the vector field $\Phi$}\label{minduefig}
\end{figure}

Let us denote by $Q$ the dashed stripe in Figure \ref{minduefig}, by $A_1$ the set enclosed by $Q$ and the cut $\Sigma$ and by $A_2$ the complement of $A_1$ with respect to $Q$.
We define a vector field $\Phi : Y_\Sigma \rightarrow \R$ prescribing its parametrization on the sheets $(D,1)$ and $(D,2)$: 
\begin{displaymath}
\Phi^1(x) = \left\{\begin{array}{ll}
(0,1)  & x\in A_1\\
(0,-1) & x\in A_2\\
0 & \text{otherwise}
\end{array}
\right.
\quad
\Phi^2(x) = \left\{\begin{array}{ll}
(0,-1)  & x\in A_1\\
(0,1) & x\in A_2\\
0 & \text{otherwise }.
\end{array}
\right.
\end{displaymath}
We verify that the unit vector field $\Phi : Y_\Sigma \rightarrow \R$ 
defined as in Figure~\ref{minduefig} is a calibration of $E$. 
First notice that $\Phi$ is an approximately regular divergence free vector field in $Y_\Sigma$. Indeed,  as a consequence of the identifications in the construction of the covering space, $\Phi$ is constant in $p^{-1}(Q)$. 
Since $\Phi$  is a piecewise constant vector field 
satisfying~\eqref{app}, its distributional divergence of is zero.
Condition (\textbf{2}) in Definition~\ref{caliconvering} is trivially satisfied. Finally 
\begin{eqnarray*}
 \int_{Y_\Sigma} \Phi \cdot D\chi_E &=& \sum_{j=1}^2 \int_{D} \Phi^j\cdot D\chi_{E^j} \\
 &=& \int_{p(\partial^\ast E)} \Phi^1 \cdot \nu_{E^1} d\Ha^1 +  \int_{p(\partial^\ast E)} \Phi^2 \cdot \nu_{E^2} d\Ha^1 \\
 &=& 2\Ha^1(p(\partial^\ast E)) \\
 &=& P(E)
\end{eqnarray*}
where we have used Lemma \ref{represent} and Proposition \ref{dueh}.

This shows, thanks to Theorem~\ref{impli}, that $E$
is a minimizer for Problem~\eqref{minpro2}.

\begin{figure}[H]
\centering
\begin{tikzpicture}
\node[inner sep=0pt] at (0,0)
    {\includegraphics[width=0.95\textwidth]{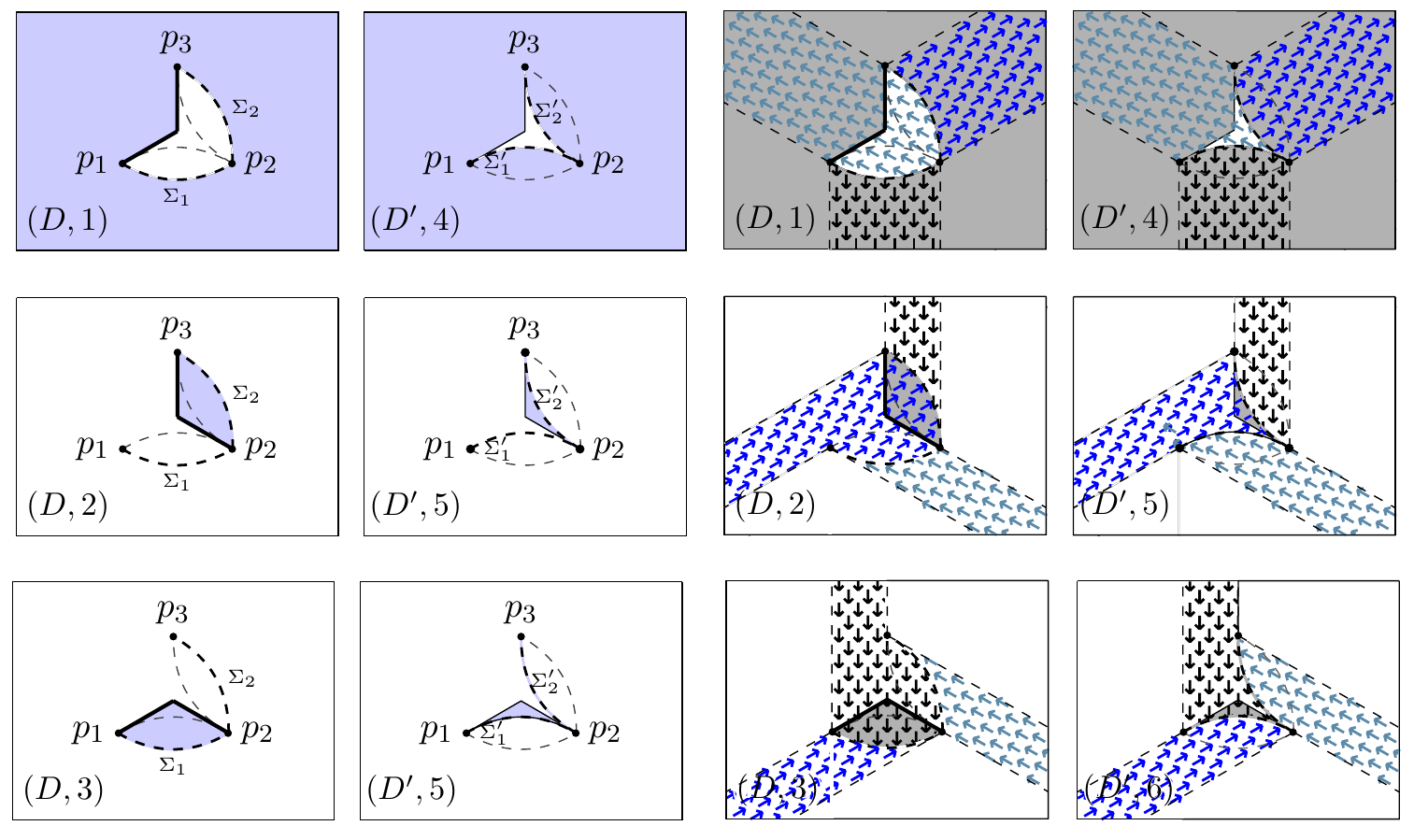}};
\end{tikzpicture}
\caption{Minimizer and calibration for three points, located at the vertices of an equilater triangle}\label{min3punti}
\end{figure}

\begin{ex}[Calibration for three points]

\end{ex}
Let us consider the case where $S$
consists of three points $p_1,p_2,p_3$.
First of all we focus our attention on the case in which the three points are the
vertices of an equilateral triangle.
We set
without loss of generality $p_1=(-\sqrt{3}/2,-1/2), \,p_2=(\sqrt{3}/2,-1/2)$ and $p_3=(0,1)$.

The set $E_{min} \in \mathscr{P}_{ constr}(Y_\Sigma)$ 
(constructed from the minimal triple junction
connecting $p_1,p_2,p_3$ following the procedure of Proposition~\ref{costruzione})
is colored in Figure~\ref{min3punti}.
We define $\Phi$ as in Figure~\ref{min3punti}. The vectors
represented by the arrows are the following:
\begin{equation}\label{calitripuntoeq}
\Phi^1 = (-1, 1/\sqrt{3}) , \quad \Phi^2 = (1, 1/\sqrt{3}),\quad \Phi^3 = (0, -2/\sqrt{3}) \,.
\end{equation}
It is easy to check that the conditions in Definition~\ref{caliconvering} are satisfied.

\begin{rem}
Notice that the calibration for three points $p_1,p_2$ and $p_3$ which are the 
vertices of a triangle with all angles of amplitude less or equal than $120$ degrees
is the same (up to a rotation and minor modifications of the extension outside the cuts and the convex envelope of the points) 
of the calibration for the equilateral triangle that we have just explicitly shown. Indeed in this case the minimal Steiner network is again the union of three segments (possibly with different lenghts) meeting in a triple junction with angles of $120$ degrees.
\end{rem}

Hence, it remains to consider the cases in which the three points of $S$
form a triangle with one angle greater or equal than $120$ degrees.
For simplicity  let $d(p_1,p_2)=d(p_2,p_3)$ and $\alpha$ be the angle between
the segment $\overline{p_1p_2}$ (respectively $\overline{p_2p_3}$) and the horizontal line 
(as in Figure~\ref{posizione3punti}).

\medskip

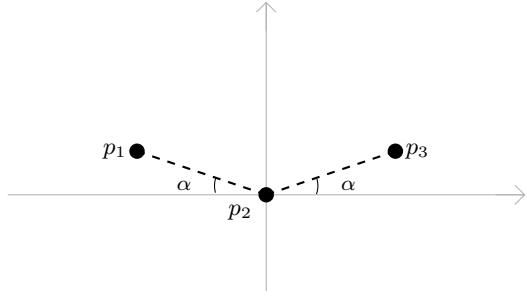
\begin{figure}[H]
\begin{center}
\begin{tikzpicture}[scale=1.7]
 \draw[black!30!white, shift={(2,-1)}, scale=0.75, rotate=-90]
(0,0)to[out= -45,in=135, looseness=1] (0.1,-0.1)
(0,0)to[out= -135,in=45, looseness=1] (-0.1,-0.1)
(0,-0.3)to[out= 90,in=-90, looseness=1] (0,0.003);
\draw[black!30!white, shift={(0,0.5)}, scale=0.75, rotate=0]
(0,0)to[out= -45,in=135, looseness=1] (0.1,-0.1)
(0,0)to[out= -135,in=45, looseness=1] (-0.1,-0.1)
(0,-0.3)to[out= 90,in=-90, looseness=1] (0,0.003);
\draw[black!30!white]
(0,-1.75)--(0,0.5)
(-2,-1)--(0,-1)
(2,-1)--(0,-1);
\draw[thick, dashed]
(-1,-0.66)--(0,-1)
(1,-0.66)--(0,-1);
\fill[black](0,-1) circle (1.7pt);
\fill[black](-1,-0.66) circle (1.7pt);
\fill[black](1,-0.66) circle (1.7pt);
 \draw[color=black,scale=0.1,domain=1.3: 1.9,
smooth,variable=\t,shift={(2,-9.3)},rotate=0]plot({2.*sin(\t r)},
{2.*cos(\t r)}) ; 
 \draw[color=black,scale=0.1,domain=1.3: 1.9,
smooth,variable=\t,shift={(-2,-9.3)},rotate=180]plot({2.*sin(\t r)},
{2.*cos(\t r)}) ; 
\path[font=\scriptsize]
(-0.5,-0.925)node[left]{$\alpha$}
(0.5,-0.925)node[right]{$\alpha$};
\path[font=\footnotesize]
(-1,-0.66)node[left]{$p_1$}
(-0.2,-1)node[below]{$p_2$}
(1,-0.66)node[right]{$p_3$};
\end{tikzpicture}
\end{center}
\caption{The set $S=\{p_1,p_2,p_3\}$ when the three points are the vertices
of a triangle with one angle greater that $120$ degrees.}\label{posizione3punti}
\end{figure}

It is well known that in this case the Steiner configuration 
that connects the points $p_1,p_2,p_3$
reduces to the two segments $\overline{p_1p_2}$ and $\overline{p_2p_3}$.
Again we construct the set
$E_{min} \in \mathscr{P}_{ constr}(Y_\Sigma)$  from the minimal Steiner configuration
following the procedure of Proposition~\ref{costruzione}.

\begin{figure}[H]
\begin{center}
\begin{tikzpicture}
\node[inner sep=0pt] at (0,0)
    {\includegraphics[width=0.485\textwidth]{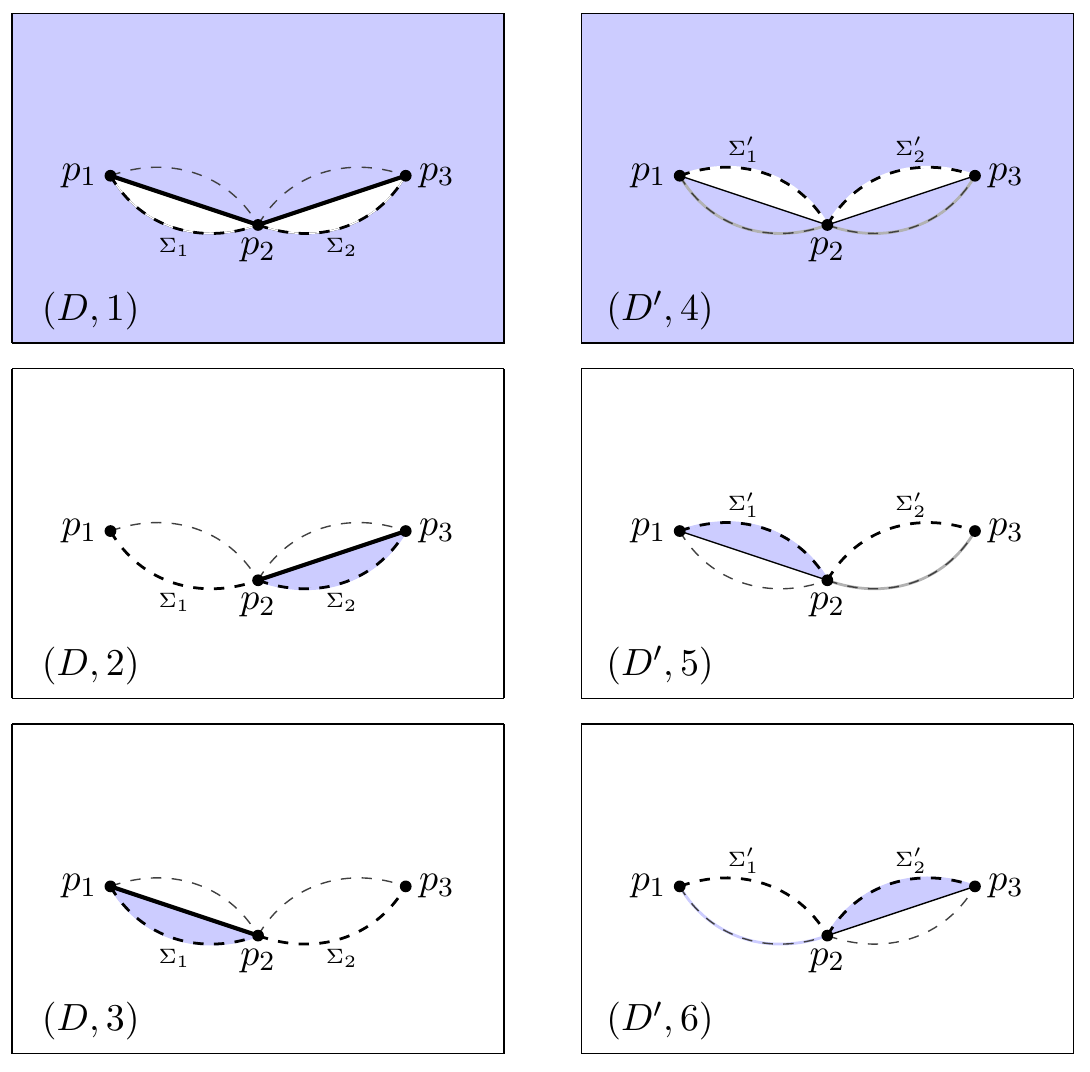}};
\end{tikzpicture}\quad
\begin{tikzpicture}
\node[inner sep=0pt] at (0,0)
    {\includegraphics[width=0.485\textwidth]{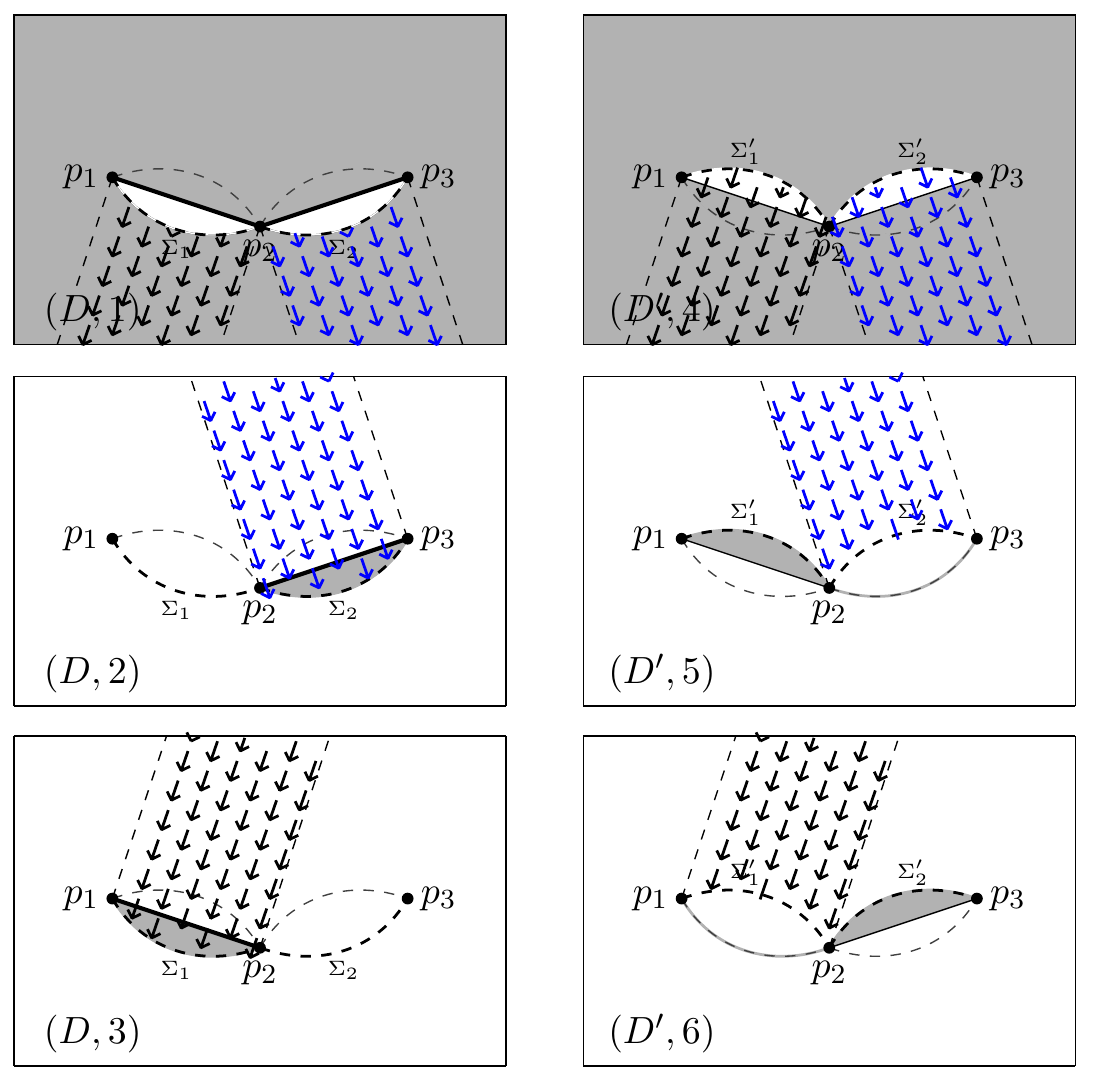}};
\end{tikzpicture}
\end{center}
\caption{Minimizer and calibration for three points}\label{min3deg}
\end{figure}

The calibration, depending on the fixed angle $\alpha \in (0,\pi/6)$, is the following (see Figure \ref{min3deg}):
\begin{equation}\label{caliangologrande}
\Phi^1 = (0, 0) , \quad \Phi^2 = (2\sin\alpha, -2\cos\alpha),\quad \Phi^3 = (-2\sin\alpha, -2\cos\alpha) \,.
\end{equation}

\begin{rem}
For $\alpha=\pi/6$ the calibration~\eqref{caliangologrande}
coincides, up to a rotation and a translation (see Remark \ref{const}), to the one for the triangle~\eqref{calitripuntoeq}.
\end{rem}

\begin{figure}[H]
\centering
\begin{tikzpicture}
\node[inner sep=0pt] at (0,0)
    {\includegraphics[width=1\textwidth]{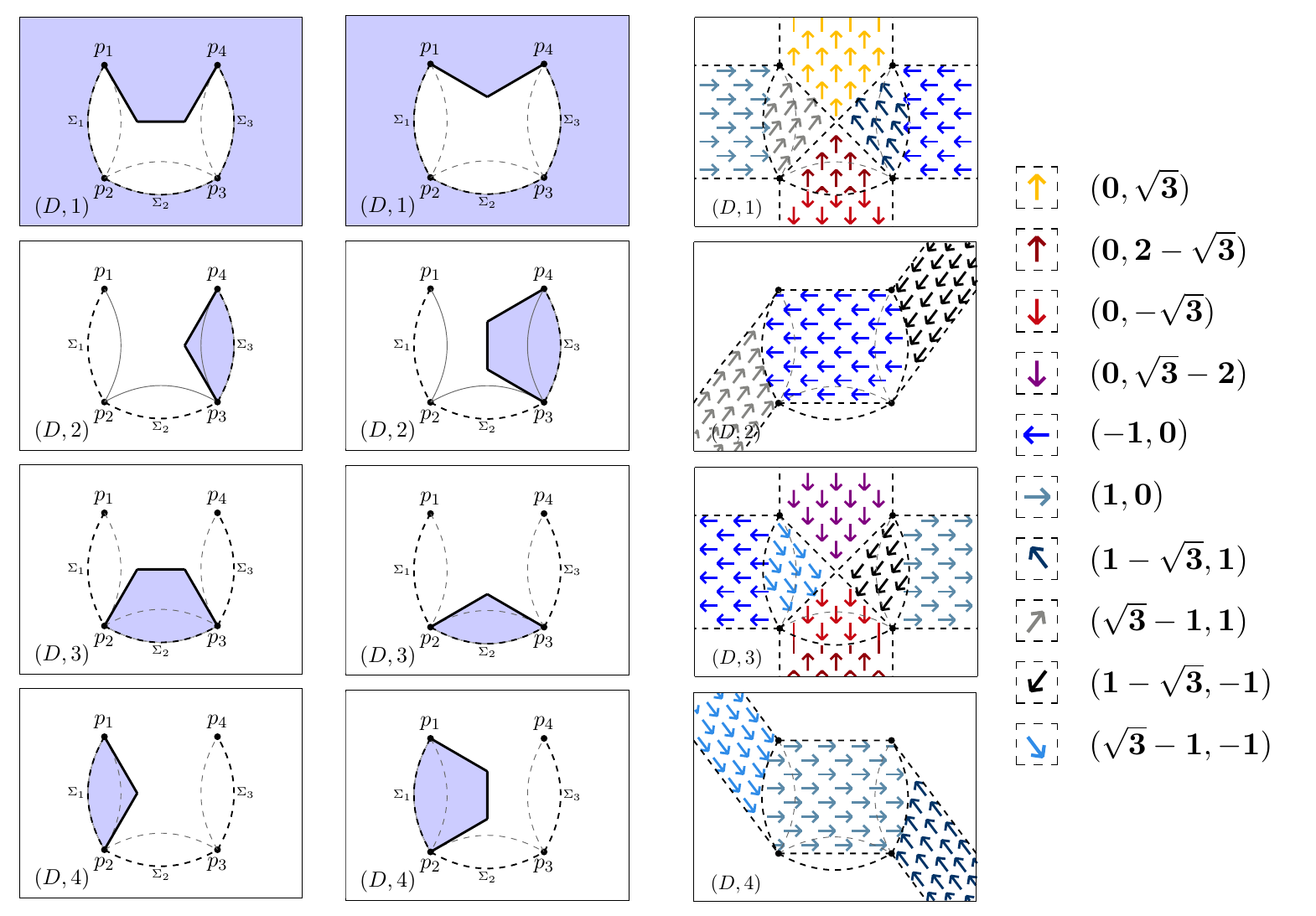}};
\end{tikzpicture}
\caption{Minimizers and calibration for the vertices of a square}\label{cal4}
\end{figure}
\begin{ex}[Calibration for the vertices of a square]

\end{ex}
Given $S=\{p_1,p_2,p_3,p_4\}$ located at the vertices of a square, the Steiner problem does not admit a unique solution. 
Therefore a calibration (if it exists) must calibrate all the minimizers.

The two candidates minimizers $E_{\min,1}, E_{\min,2}$ for Problem \eqref{minpro2} 
are shown in Figure~\ref{cal4} on the left (we draw only the sheets $(D,i)$ for $i=1,\ldots,4$).
The calibration for $E_{min, 1}$ and $E_{min, 2}$ is defined as in Figure~\ref{cal4} on the right.  
Again, we draw the vector field only in the sheets $(D,i)$ for $i=1,\ldots,4$ and we employ the usual convention that where the vector field is not written it is equal to $(0,0)$.
The reader can easily verify that the vector field defined in this way is a calibration on $Y_\Sigma$ for both the minimizers $E_{min, 1}$ and $E_{min, 2}$.

%

\section{Calibrations in families}\label{famiglie}
The aim of this section is to introduce a weaker 
definition of calibration. 
The idea is finding a way to divide the set of competitors 
in different families
defining an appropriate notion of calibration in each family
with a  weaker condition (\textbf{2}). 
Then we calibrate the minimizers of each family separately and in conclusion we compare the energy
of the minimizers to find the explicit solutions of Problem~\eqref{minpro2}. 

\medskip

\begin{dfnz}
Let $\mathcal{J} \subset \{1,\ldots,m\}\times \{1,\ldots,m\}$ 
be a subset of the Cartesian product of the indices. 
Given $E \in \mathscr{P}_{constr}(Y_\Sigma)$ we define
\begin{equation*}
E^{i,j}:= \partial^\ast E^i\cap \partial^\ast E^j
\end{equation*}
and
\begin{equation}\label{famiglia}
\mathcal{F}(\mathcal{J}) := \{E \in \mathscr{P}_{constr}(Y_\Sigma):
 \Ha^1(E^{i,j}) = 0 \mbox{ for every } (i,j) \in \mathcal{J}\}.
\end{equation}
\end{dfnz}

\begin{dfnz}[Calibration in the family $\mathcal{F}(\mathcal{J})$]\label{partcal}
Given $E \in \mathscr{P}_{constr}(Y_\Sigma)$, 
a calibration for $E$ in $\mathcal{F}(\mathcal{J})$ 
is an \emph{approximately regular} vector field $\Phi :Y_{\Sigma}\to\mathbb{R}^2$ such that
\begin{enumerate}
\item $\div\Phi=0$ (in the sense of the distributions);
\item $|\Phi^i(x) - \Phi^j(x)| \leq 2$ for every $i,j = 1,\ldots m$ 
such that $(i,j) \notin \mathcal{J}$ and for every $x\in D$;
\item $\displaystyle \int_{Y_\Sigma} \Phi \cdot D\chi_{E}= P(E)$.
\end{enumerate}
\end{dfnz}

\begin{prop}\label{implipar}
Given $\mathcal{J}$ as above and $E\in\mathcal{F}(\mathcal{J})$, 
if $\Phi :Y_{\Sigma}\to\mathbb{R}^2$ is a calibration for $E$ in the family $\mathcal{F}(\mathcal{J})$, then
\begin{equation*}
P(E)\leq P(F)
\end{equation*}
for every $F\in \mathcal{F}(\mathcal{J})$. In particular $E$ minimizes the perimeter in the class $\mathcal{F}(\mathcal{J})$.
\end{prop}
\begin{proof}
The proof is similar to that of Theorem~\ref{impli}, and it is omitted.
\end{proof}

We want to use Definition \ref{partcal}  to 
validate the minimality of a candidate minimizer for the Steiner problem.
To this aim we need to assign any competitor in $\mathscr{P}_{constr}(Y_\Sigma)$ to at least one family.
Notice that there exist sets $E\in \mathscr{P}_{constr}(Y_\Sigma)$
such that $ \Ha^1(E^{i,j})>0$  for every couple of indices $(i,j)\in  \{1,\ldots,m\}\times \{1,\ldots,m\}$,
hence it is not possible to cover $\mathscr{P}_{constr}(Y_\Sigma)$  
with non-trivial families $\mathcal{F}(\mathcal{J}_i)$.
For this reason we restrict to $\mathscr{P}^T_{constr}(Y_\Sigma)$.

\begin{dfnz}
We call $T$ the set of all connected networks without loops (see Definition~\ref{conetw}).
Moreover we call $\mathscr{P}_{constr}^T(Y_\Sigma)$ 
the set of all  $E\in\mathscr{P}_{constr}(Y_\Sigma)$ such that
$\overline{p(\partial^\ast E)}$ is an element of $T$.
\end{dfnz}
Consider the following:
\begin{prob}\label{stp}
Given $S$ a finite sets of points in $\mathbb{R}^2$
we look for a network  in $T$ with minimal length that connects
the points of $S$.
\end{prob}
It is well known that Problem~\ref{stp} is equivalent to the Steiner problem defined in~\eqref{ste} (see, for instance \cite{paopaostepanov}).
Therefore if we define $\mathscr{A}^T_{constr}(S) = \inf \left\{P(E) : E\in\mathscr{P}^T_{constr}(Y_\Sigma)\right\}$ we infer, thanks to Theorem~\ref{samelength}, that $\mathscr{A}^T_{constr}(S) = \mathscr{A}_{constr}(S) 
$.

\begin{prop}\label{mincalifam}
Suppose that there exists $\mathcal{J}_1, \ldots, \mathcal{J}_N \subset \{1,\ldots,m\}\times \{1,\ldots,m\}$ such that
\begin{equation}\label{parti}
\mathscr{P}^T_{constr}(Y_\Sigma) \subseteq \bigcup_{i = 1}^N \mathcal{F}(\mathcal{J}_i).
\end{equation}
If for every $i=1,\ldots,N$ there exists a calibration
$\Phi_i$  for $E_i$ in $\mathcal{F}(\mathcal{J}_i)$, then 
\begin{equation}\label{mineq}
\mathscr{A}_{\rm{constr}}(S)= \min\{P(E_i): i= 1,\ldots,N\}.
\end{equation}
In other words the set $E_i$ with less perimeter is the absolute minimizer of Problem~\eqref{minpro2}.
\end{prop}
\begin{proof}
Fix $F\in \mathscr{P}^T_{constr}(Y_\Sigma)$. 
Thanks to~\eqref{parti} there exists at least one $i\in \{1,\ldots,N\}$ such that 
$F\in \mathcal{F}(\mathcal{J}_i)$. Proposition~\ref{implipar} implies that 
$P(E_i)\leq P(F)$.
Thus 
\begin{equation*}
\min\{P(E_i): i= 1,\ldots,N\} \leq P(F)\,,
\end{equation*}
that is~\eqref{mineq}.
\end{proof}
\begin{rem}
In Subsection~\ref{aaaah} we will see that it is relevant \emph{how coarse} is the decomposition of $\mathscr{P}^T_{constr}(Y_\Sigma)$ in families $\mathcal{F}(\mathcal{J}_i)$. Indeed if the cardinality of $\mathcal{J}_i$ is small, then one needs less families to cover $\mathscr{P}^T_{constr}(Y_\Sigma)$, but the task of finding an explicit calibration in $\mathcal{F}(\mathcal{J}_i)$ results more challenging. 
\end{rem}

\subsection{Examples}\label{aaaah}
Suppose that the finite set
$S=\{p_1,\ldots,p_m\}$ consists of $m$ points 
located on the boundary of a smooth, open, convex set $A\subset \mathbb{R}^2$.
For simplicity we label the points on the boundary of $A$ in a anticlockwise order
and from now on we consider the indices $i=1,\ldots, m$ cyclically identified modulus $m$.

It is not restrictive to 
suppose that each competitor $\Gamma\in T$ is contained in $A$, so that
$T$ induces a partition of 
$A\setminus \Gamma$ in $m$ connected 
sets $\{A_\Gamma^1,\ldots,A_\Gamma^m\}$
labelled in such a way that $\{p_i,p_{i+1}\}\subset\partial A^{m+1-i}$.
Calling $A_\Gamma^{i,j} := \partial A_\Gamma^{i} \cap \partial A_\Gamma^{j}$ for every $i,j= 1,\ldots,m$, the Steiner problem can be rephrased as 
\begin{equation}\label{ste2}
\min\left\{\sum_{i<j} \Ha^1(A_\Gamma^{i,j}): \Gamma\in T\right\}\,.
\end{equation}

We now suggest a general and explicit way to cover $\mathscr{P}^T_{constr}(Y_\Sigma)$
with families $\mathcal{F}(\mathcal{J}_i)$ in order to use the notion of calibration in $\mathcal{F}(\mathcal{J}_i)$ and Proposition~\ref{mincalifam}
to show explicit solution of Problem~\eqref{minpro2}.

\begin{lemma}\label{1split}
Consider $\Gamma \in T$ inducing the partition $\{A_\Gamma^1,\ldots,A_\Gamma^m\}$
and suppose that there exists $i\neq j$ such that  $A^{i,j}_\Gamma \neq \emptyset$. 
Then for every $0\leq k_1<i<k_2<j \leq m$ (or $0\leq j<k_1<i<k_2 \leq m$) we have that
\begin{equation*}
\Ha^1(A_\Gamma^{k_1,k_2})=0\,.
\end{equation*}
\end{lemma}
\begin{proof}
By contradiction it is enough to notice that if $\Ha^1(A_\Gamma^{k_1,k_2})>0$, then the interior of $\overline{A_\Gamma^{k_1} \cup A_\Gamma^{k_2}}$ is an open connected set that separates $A_\Gamma^i$ and $A_\Gamma^j$. Hence we infer $A_{\Gamma}^{i,j} = \emptyset$.
\end{proof}

We construct the covering space $Y_\Sigma$
choosing an admissible pair of cuts in the following way:
the cut $\Sigma'$ coincides with $\partial A$ and the cut $\Sigma$ lies outside $A$.
Then, thanks to Proposition~\ref{costruzione}, it is possible to
associate to the network $\Gamma$, and hence to the partition $\{A_\Gamma^1,\ldots,A_\Gamma^m\}$,
a set $E_{\Gamma}$ in the covering space $Y_\Sigma$ (simply setting
 $E_\Gamma^j=A_\Gamma^j$) such that
\begin{equation}\label{poi}
p(\overline{\partial^\ast E_\Gamma}) =\Gamma= \bigcup_{i,j}A_\Gamma^{i,j} \quad \mbox{and} 
\quad E^{i,j}_\Gamma = A_\Gamma^{i,j}. 
\end{equation}

Thanks to~\eqref{poi} Lemma~\ref{1split} 
is trivially true replacing $A_\Gamma^{i,j}$ with $E_\Gamma^{i,j}$.

We define
\begin{equation*}
\mathcal{F}_{i,j} = \{\Gamma \in T: A^{i,j}_\Gamma \neq \emptyset\}\,.
\end{equation*}
It is easy to see that $T$ can be covered in the following way:
\begin{equation}\label{spliiiiit}
T = \bigcup_{k=2}^{\lfloor\frac{m}{2}\rfloor} \bigcup_{ |i-j| = k}  \mathcal{F}_{i,j}\,.
\end{equation}
This covering induces automatically 
a covering of $\mathscr{P}_{constr}(Y_\Sigma)$.
Consider for instance a family $\mathcal{F}_{i,i+\lfloor\frac{m}{2}\rfloor}$ 
for a fixed $i$ in $\{1,\ldots,m\}$.
Thanks to Lemma~\ref{1split} 
we have that
\begin{displaymath}
\Ha^1(E^{k,l}_{\Gamma}) = 0 
\end{displaymath}
for all $(k,l) \in\{i+1,\ldots,i+\lfloor\frac{m}{2}\rfloor-1\}\times\{i+\lfloor\frac{m}{2}\rfloor+1,i-1\}$. 
This property defines a family $\mathcal{F}(\mathcal{J})$ according to~\eqref{famiglia}
with $\mathcal{J}=\{i+1,\ldots,i+\lfloor\frac{m}{2}\rfloor-1\}\times\{i+\lfloor\frac{m}{2}\rfloor+1,i-1\}$.

\begin{ex}[Calibration for the vertices of the regular pentagon]\label{penta}

\end{ex}
Consider $S=\{p_1,p_2,p_3,p_4,p_5\}$ located at the vertices of a regular pentagon. 
First we divide the elements of $\mathscr{P}^T_{constr}(Y_\Sigma)$ in families.
According to~\eqref{spliiiiit}, we cover $T$ with five families $\mathcal{F}_{i,j}$ as follows:
\begin{equation*}
T= \bigcup_{|i-j| = 2}  \mathcal{F}_{i,j}\,.
\end{equation*} 
Then we split again each family $\mathcal{F}_{i,j}$ 
in two subfamilies $\mathcal{F}^1_{i,j}$ and $\mathcal{F}^2_{i,j}$ defined as
\begin{equation*}
\mathcal{F}^1_{i,j} = \mathcal{F}_{i,j} \cap \mathcal{F}_{i,j+1} \quad \mbox{and} \quad \mathcal{F}^2_{i,j} = 
\mathcal{F}_{i,j} \cap \mathcal{F}_{i-1,j}\,.
\end{equation*}
This produces in principle $10$ families, but it is easy to see that $\mathcal{F}^k_{i,j}=\mathcal{F}^{k'}_{i',j'}$
for some $i,j,i',j'\in\{1,2,3,4,5\}$ and $k,k'\in\{1,2\}$.
Hence, 
we obtain 
\begin{equation*}
\mathscr{P}^T_{constr}(Y_\Sigma)=\bigcup_{i=1}^5\mathcal{F}(\mathcal{J}_i)\,,
\end{equation*}
with 
\begin{align*}
\mathcal{J}_1=\{(1,3),(1&,4),(2,4)\},\quad\mathcal{J}_2=\{(1,3),(1,4),(3,5)\},
\quad\mathcal{J}_3=\{(1,3),(2,5),(3,5)\}, \\
&\mathcal{J}_4=\{(1,4),(2,4),(2,5)\}, \quad\mathcal{J}_5=\{(2,4),(2,5),(3,5)\}\,.
\end{align*}

It is known that the Steiner problem for $S$ has $5$ minimizers 
$\mathscr{S}_i$ for $i = 1,\ldots,5$ (obtained by rotation one from the other).
Denoted by $E_{min,i}\in \mathscr{P}^T_{constr}(Y_\Sigma)$ for $i = 1,\ldots,5$
the sets associated with the minimizers of the Steiner problem,
it is easy to see that 
that $E_{min,i}\in\mathcal{F}(\mathcal{J}_i)$ for $i = 1,\ldots,5$.
Our aim is to prove that $E_{min,i}$ is a minimizer in $\mathcal{F}(\mathcal{J}_i)$
constructing a vector field $\Phi_i$ that is a calibration for $E_{min,i}$ in $\mathcal{F}(\mathcal{J}_i)$.

On the left of Figure~\ref{cali5} is shown the set 
$E_{min,5}\in \mathscr{P}^T_{constr}(Y_\Sigma)$, on the right 
a calibration for $E_{min,5}$ in $\mathcal{F}(\mathcal{J}_{5})$. 
The vector field represented by the arrows is the following:
\begin{align*}
\Phi^1=(0,0),\quad \Phi^2 =(2,0), \quad \Phi^3=(1,-\sqrt{3}),\quad 
\Phi^4(-1,-\sqrt{3}),\quad \Phi^5=(-2,0)\,
\end{align*}
and it is easy to verify that it is indeed a calibration for $E_{min,5}$ in $\mathcal{F}(\mathcal{J}_{5})$. 

As the minimizers $\mathscr{S}_i$ with $i = 1,\ldots,5$ for the 
Steiner problem are
 obtained by rotation one from the other, 
it is easy to construct 
for $E_{min,i}$  (with $i=1,2,3,4$)
a calibration in $\mathcal{F}(\mathcal{J}_i)$ similar to
the one for  $E_{min,5}$ in $\mathcal{F}(\mathcal{J}_5)$.

To summarize we have split the set $\mathscr{P}_{constr}^T(Y_\Sigma)$ and
we have exhibited a calibration in each family for the corresponding $E_{min,i}$. 
As $\mathcal{H}^1(\mathscr{S}_i)=\mathcal{H}^1(\mathscr{S}_j)$ for $i,j\in\{1,\ldots,5\}$,
thanks to Proposition~\ref{dueh}
we have also that $P(E_{min,i}) = P(E_{min,j})$ for every $i,j\in \{1,\ldots,5\}$. 
Thus applying Proposition~\ref{mincalifam} we infer that $E_{min,i}$ are minimizers 
of Problem~\eqref{minpro2} for every $i=1\ldots,5$, as we wanted to prove.

\begin{figure}[H]
\centering
\begin{tikzpicture}
\node[inner sep=0pt] at (0,0)
    {\includegraphics[width=0.44\textwidth]{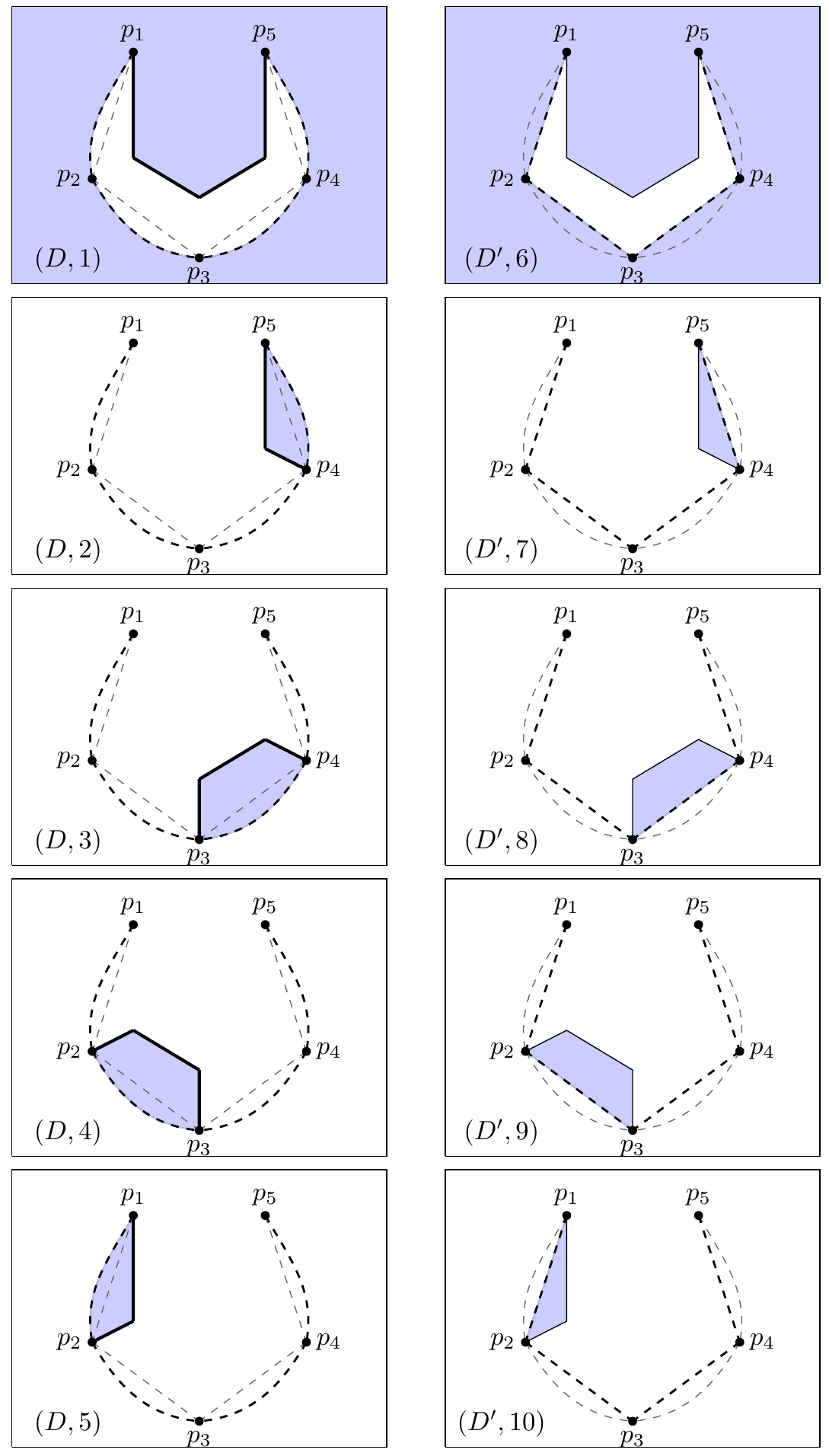}};
\end{tikzpicture}\quad\;
\begin{tikzpicture}
\node[inner sep=0pt] at (0,0)
    {\includegraphics[width=0.44\textwidth]{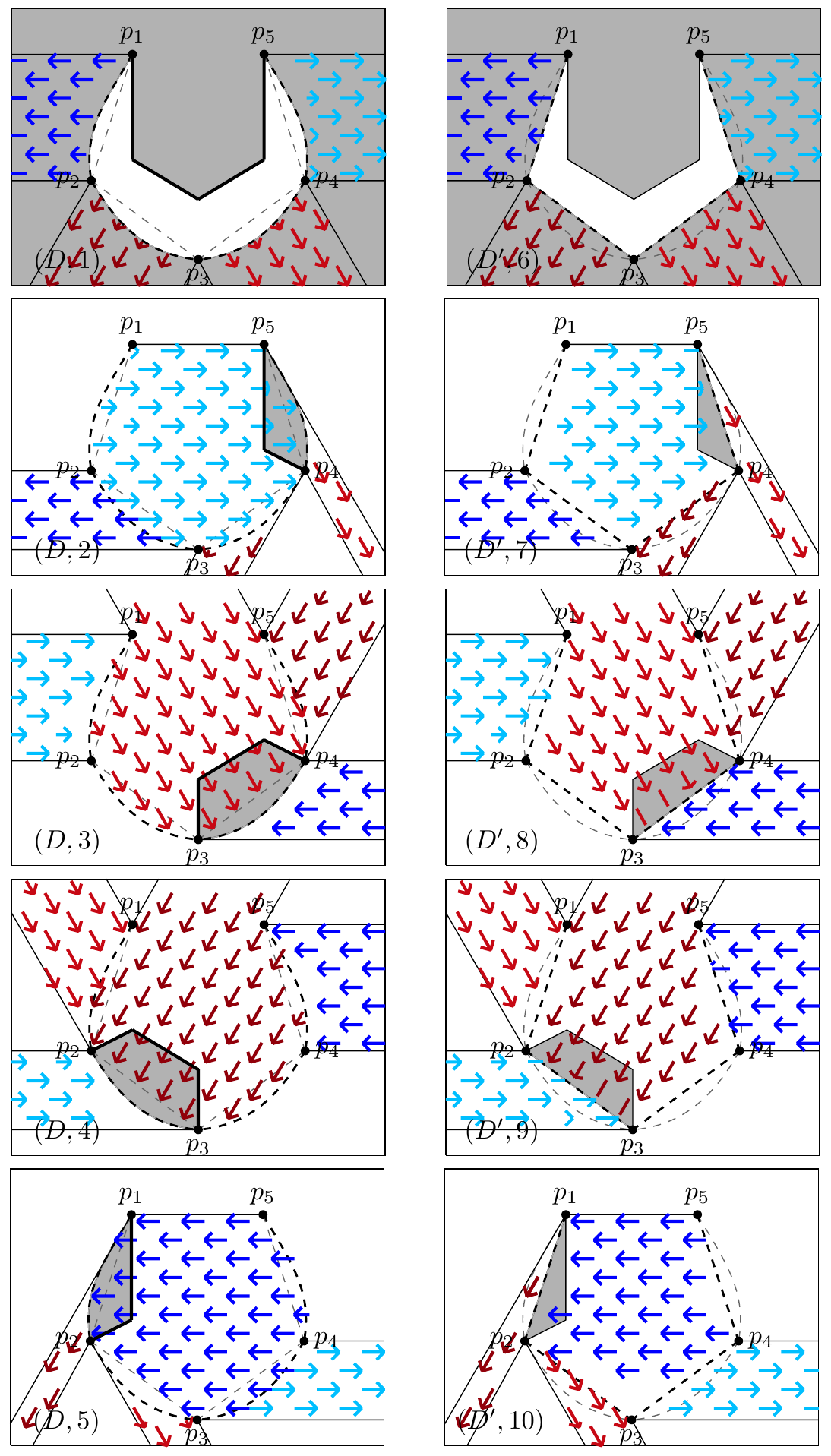}};
\end{tikzpicture}\quad\;
\caption{The minimizer $u_{min,5}$ for five points at the vertices of a regular pentagon and a calibration for the family $\mathcal{F}(\mathcal{J}_5)$}\label{cali5}
\end{figure}

\begin{rem}
In~\cite{coveringbis} we prove that 
if  $\Phi:Y\to\mathbb{R}^2$ is a calibration for $E\in\mathcal{P}_{constr}$,
then $E$ is a minimizer not only among all (constrained) finite perimeter sets, 
but also in the larger class of finite linear combinations of characteristic 
functions of (constrained) finite perimeter sets.
Then if there exists 
an element of this larger class with strictly 
less energy of the minimizer of Problem~\eqref{minpro2}, 
a calibration for such a minimizer cannot exist.
This is the case when $S=\{p_1,\ldots,p_5\}$ with $p_i$ the vertices of a regular pentagon.
This counterexample can be constructed adapting 
the example  by Bonafini~\cite{bonafini} in the framework
of rank one tensor valued measures to our setting.
Hence the tool of the calibration in families is necessary, in this case, to prove the minimality of the candidate by a calibration argument.
\end{rem}

\begin{ex}[Calibration for the vertices of the regular hexagon]

\end{ex}

We fix the points of $S$ as the vertices of a regular hexagon  in the following way:
$p_1=(-1/2,\sqrt{3}/2)$, $p_2=(-1,0)$, $p_3=(-1/2,-\sqrt{3}/2)$, $p_4=(1/2,-\sqrt{3}/2)$,
$p_5=(1,0)$, $p_6=(1/2,\sqrt{3}/2)$.

As in Example~\ref{penta} we start covering
$\mathscr{P}^T_{constr}(Y_\Sigma)$  with explicit families $\mathcal{F}(\mathcal{J})$ of competitors.
From~\eqref{spliiiiit} we get
\begin{equation*}
T = \left(  \bigcup_{ |i-j| = 3}  \mathcal{F}_{i,j} \right)  \cup 
\left( \bigcup_{ |i-j| = 2}  \mathcal{F}_{i,j}\right)\,.
\end{equation*}

For given $i,j\in \{1,\ldots,6\}$ such that $|i-j|=3$ we further split $\mathcal{F}_{i,j}$ in four classes $(\mathcal{F}^k_{i,j})_{k=1,\ldots,4}$ as follows: 
\begin{displaymath}
\mathcal{F}^1_{i,j} = \mathcal{F}_{i,j} \cap (\mathcal{F}_{i,j-1}  \cup \mathcal{F}_{i,j+1}), \qquad \mathcal{F}^2_{i,j} = \mathcal{F}_{i,j} \cap (\mathcal{F}_{i+1,j}  \cup \mathcal{F}_{i-1,j})\, ,
\end{displaymath}
\begin{displaymath}
\mathcal{F}^3_{i,j} = \mathcal{F}_{i,j} \cap (\mathcal{F}_{i,j-1}  \cup \mathcal{F}_{i+1,j}), \qquad \mathcal{F}^4_{i,j} = \mathcal{F}_{i,j} \cap (\mathcal{F}_{i-1,j}  \cup \mathcal{F}_{i,j+1})\, .
\end{displaymath}

As in the previous example,
thanks to Lemma~\ref{1split} applied to the families $ \mathcal{F}^k_{i,j}$
with $i,j\in \{1,\ldots,6\}$ and $k\in\{1,\ldots,4\}$, we can associate the 
respectively families $\mathcal{F}(\mathcal{J}_i)$ with $J_i$ defined as follows:
 \begin{align*}
&\mathcal{J}_1=\{(2,4),(2,5),(2,6),(3,5),(3,6),(4,6)\},
&\mathcal{J}_2=\{(1,3),(1,4),(1,5),(3,5),(3,6),(4,6)\},\\
&\mathcal{J}_3=\{(1,4),(1,5),(2,4),(2,5),(2,6), (4,6)\},
&\mathcal{J}_4=\{(1,3),(1,5),(2,5),(2,6),(3,5),(3,6)\},\\
&\mathcal{J}_5=\{(1,3),(1,4),(2,4),(2,6),(3,6),(4,6)\},
&\mathcal{J}_6=\{(1,3),(1,4),(1,5),(2,4),(2,5),(3,5)\},\\
&\mathcal{J}_7=\{(1,5),(2,4),(2,5),(2,6),(3,5),(3,6)\},
&\mathcal{J}_8=\{(1,3),(1,4),(2,6),(3,5),(3,6),(4,6)\},\\
&\mathcal{J}_9=\{(1,3),(1,4),(1,5),(2,4),(2,5),(4,6)\},
&\mathcal{J}_{10}=\{(1,3),(2,5),(2,6),(3,5),(3,6), (4,6)\},\\
&\mathcal{J}_{11}=\{(1,3),(1,4),(1,5),(2,4),(3,6),(4,6)\},
&\mathcal{J}_{12}=\{(1,4),(1,5),(2,4),(2,5),(2,6),(3,5)\}.\\
\end{align*}
We notice that the families $\mathcal{F}(\mathcal{J}_i)$ with $i\in\{1,\ldots,6\}$
can be obtained from the first one by a cyclic permutation of the indices.
Also the families $\mathcal{F}(\mathcal{J}_8)$ and $\mathcal{F}(\mathcal{J}_9)$
can be obtained by a cyclic permutation of the indices from $\mathcal{F}(\mathcal{J}_7)$,
and the same holds for the families $\mathcal{F}(\mathcal{J}_{10})$, $\mathcal{F}(\mathcal{J}_{11})$,
and $\mathcal{F}(\mathcal{J}_{12})$ (see Figure \ref{familieshex}).

Consider now the case in which  $|i-j|=2$ for given $i,j\in \{1,\ldots,6\}$.
Here the situation is easier, as we find two families $\mathcal{F}_{i,j}$
and it is not necessary to consider a further refinement of the classes.
In terms of $\mathcal{F}(\mathcal{J}_i)$ we get
\begin{equation*}
\mathcal{J}_{13}=\{(1,4),(2,4),(2,5),(2,6),(3,6),(4,6)\}\, ,\;
\mathcal{J}_{14}=\{(1,3),(1,4),(1,5),(2,5),(3,5), (3,6)\}\, ,
\end{equation*}
where again $\mathcal{F}(\mathcal{J}_{14})$ is obtained by a cyclic permutation of the indices from 
$\mathcal{F}(\mathcal{J}_{13})$.
In conclusion the subdivision in families is as follows: 
\begin{equation*}
\mathscr{P}^T_{constr}(Y_\Sigma)=\bigcup_{i=1}^{14}\mathcal{F}(\mathcal{J}_i)\,.
\end{equation*}

\begin{figure}[H]
\centering
\begin{tikzpicture}
\node[inner sep=0pt] at (0,0)
    {\includegraphics[width=0.231\textwidth]{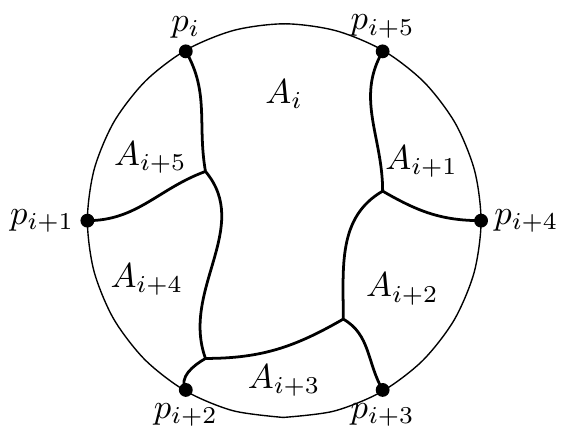}};
\end{tikzpicture}\quad
\begin{tikzpicture}
\node[inner sep=0pt] at (0,0)
    {\includegraphics[width=0.231\textwidth]{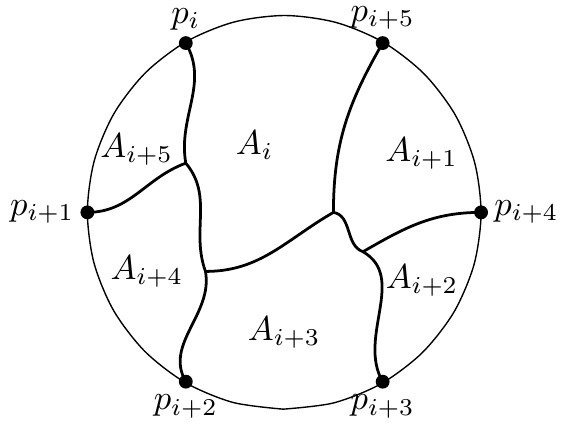}};
\end{tikzpicture}\quad
\begin{tikzpicture}
\node[inner sep=0pt] at (0,0)
    {\includegraphics[width=0.231\textwidth]{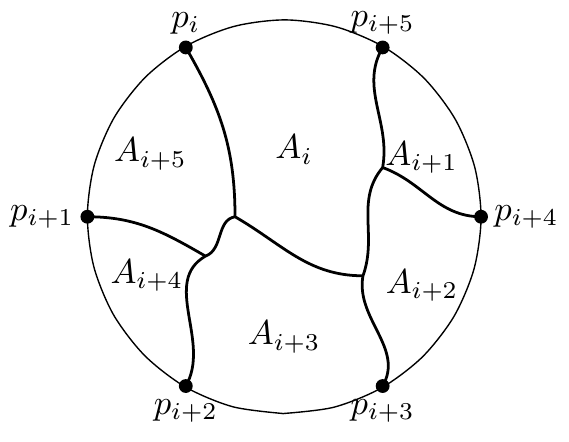}};
\end{tikzpicture}\quad
\begin{tikzpicture}
\node[inner sep=0pt] at (0,0)
    {\includegraphics[width=0.231\textwidth]{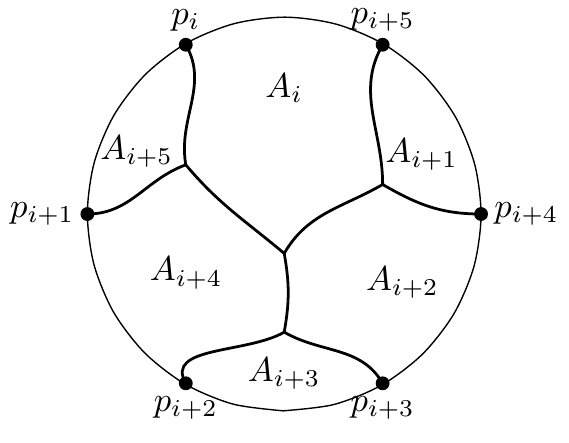}};
\end{tikzpicture}
\caption{On the left we represent the projection onto the base set of the essential boundary 
of an element of the family
$\mathcal{F}(\mathcal{J}_i)$ with $i\in\{1,\ldots,6\}$, then of the family
$\mathcal{F}(\mathcal{J}_i)$ with $i\in\{7,8,9\}$ and 
$\mathcal{F}(\mathcal{J}_i)$ with $i\in\{10,11,12\}$, on the right $\mathcal{F}(\mathcal{J}_{i})$
with $i\in\{13,14\}$. 
}\label{familieshex}
\end{figure}

It is well know that if the points of $S$ lies at the vertices of a regular hexagon,
then there are six minimizers $\mathscr{S}_i$ for the Steiner problem.
Calling $E_{\min,i}$ with $i\in\{1,\ldots,6\}$ 
the sets in $\mathscr{P}_{constr}(Y_\Sigma)$ associated to  $\mathscr{S}_i$, we have that 
$E_{\min,i}\in \mathcal{F}(\mathcal{J}_i)$ for $i\in \{1,\ldots,6\}$. 

In order to use Proposition~\ref{mincalifam}
we have to find a calibration $\Phi^i$ for an explicit set $E_{i}\in \mathcal{F}(\mathcal{J}_i)$
for every $i\in\{1,\ldots,14\}$.
The global minimizers $E_{\min,i}$ are clearly minimizers in their families,
so they are the natural candidate minimizers for the families
$ \mathcal{F}(\mathcal{J}_i)$ for $i\in\{1,\ldots,6\}$; it is more challenging to propose a minimizer for the other families.
Our candidate minimizers are shown in Figure~\ref{candidati}.

\begin{figure}[H]
\centering
\begin{tikzpicture}
\node[inner sep=0pt] at (0,0)
    {\includegraphics[width=0.231\textwidth]{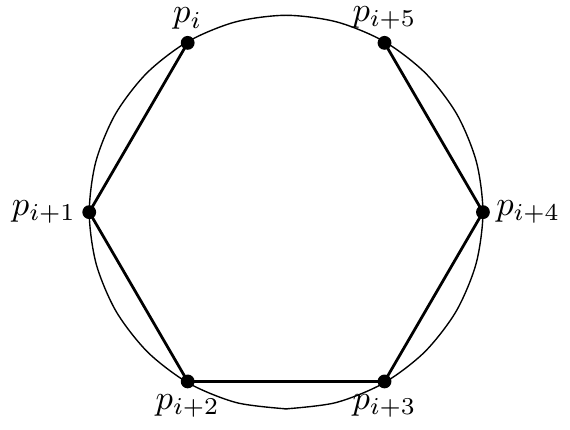}};
\end{tikzpicture}\quad
\begin{tikzpicture}
\node[inner sep=0pt] at (0,0)
    {\includegraphics[width=0.231\textwidth]{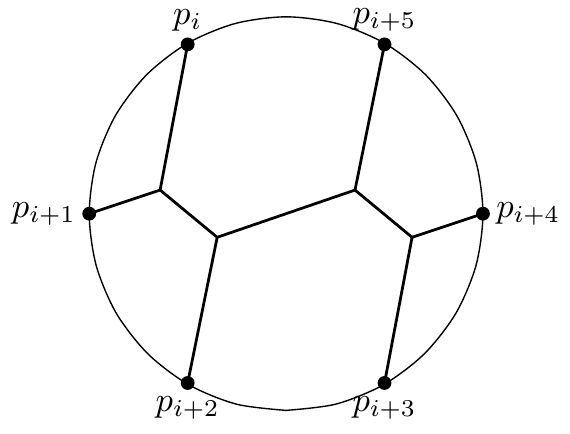}};
\end{tikzpicture}\quad
\begin{tikzpicture}
\node[inner sep=0pt] at (0,0)
    {\includegraphics[width=0.231\textwidth]{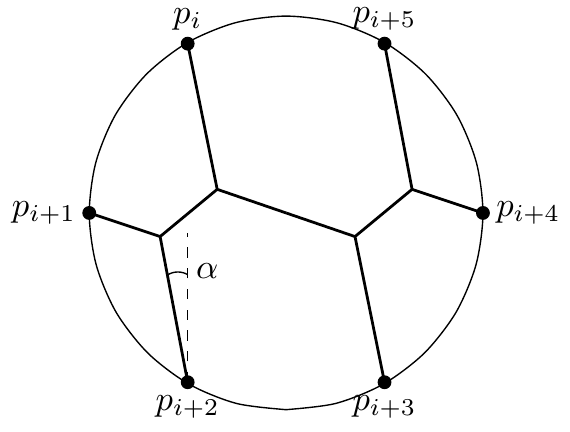}};
\end{tikzpicture}\quad
\begin{tikzpicture}
\node[inner sep=0pt] at (0,0)
    {\includegraphics[width=0.231\textwidth]{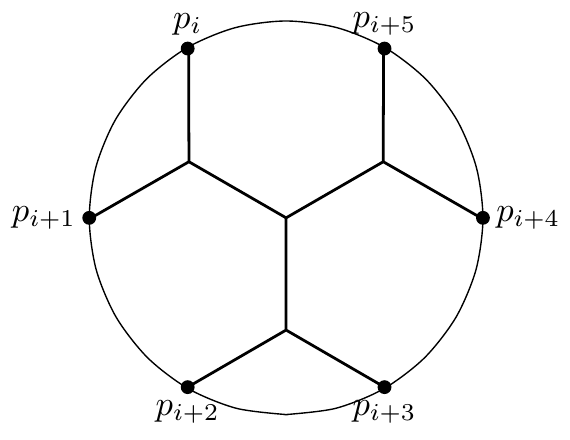}};
\end{tikzpicture}
\caption{From left to right: the projection onto the base set of the essential boundary
of candidate minimizers in the family 
$\mathcal{F}(\mathcal{J}_i)$ with $i\in\{1,\ldots,6\}$, $i\in\{7,8,9\}$
$i\in\{10,11,12\}$ and $i\in\{13,14\}$.}\label{candidati}
\end{figure}

We write explicitly the calibration $\Phi_i$ (written, as usual, by means of the pullbacks on every sheet $\Phi_i^j$) for $E_{\min,i}\in \mathcal{F}(\mathcal{J}_i)$
with $i=1,10,13$. The other calibrations are easy variants of the previous ones.
Moreover we give the expression of the vector field only inside $p^{-1}(\mbox{Conv}(S))$, as the divergence free extension in $Y_\Sigma$ can be easily achieved.
We get 
\begin{equation*}
\begin{array}{lll}
\Phi^{1}_1=(0,0)\,, &\Phi^{2}_1 =(\sqrt{3},1)\,, &\Phi^{3}_1=(\sqrt{3},-1)\,,\\
\Phi^{4}_1=(0,-2)\,,&\Phi^{5}_1=(-\sqrt{3},-1)\,,& \Phi^{6}_1=(-\sqrt{3},1)\,,\\ \\
\Phi^{1}_{10}=(0,0)\,,&\Phi^{2}_{10} =\left( \frac{3\sqrt{3}}{\sqrt{7}},\frac{1}{\sqrt{7}}\right)\,,
&\Phi^{3}_{10}=\left( \frac{2\sqrt{3}}{\sqrt{7}}\,,-\frac{4}{\sqrt{7}}\right)\,,\\
\Phi^{4}_{10}=\left( -\frac{\sqrt{3}}{\sqrt{7}}\,,-\frac{5}{\sqrt{7}}\right) \,,
&\Phi^{5}_{10}=\left( -\frac{4\sqrt{3}}{\sqrt{7}},-\frac{6}{\sqrt{7}}\right) \,, 
&\Phi^{6}_{10}=\left( -\frac{3\sqrt{3}}{\sqrt{7}},-\frac{1}{\sqrt{7}}\right) \,,\\ \\
\Phi^{1}_{13}=(0,0)\,, &\Phi^{2}_{13} =(2,0)\,, &\Phi^{3}_{13}=(1,-\sqrt{3})\,,\\
\Phi^{4}_{13}=(0,-2\sqrt{3})\,,
&\Phi^{5}_{13}=(-1,-\sqrt{3})\,, &\Phi^{6}_{13}=(-2,0)\,.
\end{array}
\end{equation*}
\begin{rem}
The vector field $\Phi_{10}$ is a calibration for the set $E\in \mathscr{P}^T_{constr}(Y_\Sigma)$
associated to the third network represented in Figure~\ref{candidati} for $i=1$.
It can be computed easily as follows: firstly we rotate the points of $S$ and the network
by a rotation of angle $\alpha = -\arctan{(\frac{1}{3\sqrt{3}})}$ with the following rotation matrix:
\begin{equation*}
R(\alpha) = \begin{bmatrix}
\frac{3\sqrt{3}}{2\sqrt{7}} & \frac{1}{2\sqrt{7}}\\
-\frac{1}{2\sqrt{7}} &\frac{3\sqrt{3}}{2\sqrt{7}}\,  \\
\end{bmatrix}
\,.
\end{equation*}
The network obtained in this way can be calibrated with the following vector field:
\begin{equation*}
\begin{array}{lll}
\widetilde{\Phi}^{1}_{10}=(0,0)\,, &\widetilde{\Phi}^{2}_{10} =(2,0)\,, &\widetilde{\Phi}^{3}_{10}=(1,-\sqrt{3})\,,\\
\widetilde{\Phi}^{4}_{10}=(-1,-\sqrt{3})\,,
&\widetilde{\Phi}^{5}_{10}=(-3,-\sqrt{3})\,, &\widetilde{\Phi}^{6}_{10}=(-2,0)\,\,.
\end{array}
\end{equation*}
Finally we rotate back with the matrix $R(-\alpha)$ 
to get $\Phi_{10}$.

\medskip

An easy computation shows that $P(E_{min,i}) < P(E_{min,j})$ for $i\in \{1,\ldots,6\}$ and $j\in \{7,\ldots,14\}$. Thus applying Proposition \ref{mincalifam} we infer that $E_{min,i}$ are minimizers of $\mathscr{A}_{constr}(S)$ for every $i=1,\ldots,6$, as we wanted to prove.
\end{rem}

\begin{rem}
The above division in families is the finest possible one.
It has the advantage that we obtain constant calibrations.
However from the numerical point of view this choice is not convenient. Indeed 
the complexity is simply shifted from solving the  non--convex
original problem to finding all the families.
It would be better to consider a more coarse division in families 
in which one can in any case find a calibration. 
\end{rem}

\section*{Appendix: Divergence theorem on $Y_\Sigma$}\label{appe}
\begin{prop}[Divergence theorem on coverings]
Consider $E,F\in \mathscr{P}_{constr}(Y_\Sigma)$ and
let $\Phi : Y_\Sigma \to \R^2$ be an approximately regular vector field such that $\div \Phi = 0$ 
(in the sense of distributions)
in $Y_\Sigma$.
Then
\begin{equation}
\int_{Y_\Sigma} \Phi \cdot D\chi_E = \int_{Y_\Sigma} \Phi \cdot D\chi_F.
\end{equation}  
\end{prop}
\begin{proof}
From~\eqref{derivata},
one gets:
\begin{equation*}
D\chi_E = \sum_{j=1}^m D\chi_{E^j} \res D  
+ \sum_{j'=m+1}^{2m} D\chi_{E^{j'}}\res (\Sigma\setminus S)\,.
\end{equation*}
Fix  $\varepsilon>0$ and call $\Omega_\varepsilon$ the $\varepsilon$--tubular neighbourhood
of $\Omega$. Using the divergence theorem for approximately regular vector fields in 
$D\subset \R^2$ (see~\cite{mumford}) and the definition of $D\chi_{E^{j'}}$ we have
\begin{align}
&  \int_{Y_\Sigma} \Phi \cdot D\chi_E
= \sum_{j=1}^m \int_{D} \Phi^j\cdot D\chi_{E^j}
+ \sum_{j'=m+1}^{2m} \int_{\Sigma} \Phi^{j'}\cdot D\chi_{E^{j'}}\nonumber\\
=&-\sum_{j=1}^m  \int_{\Sigma} \big[(\Phi^j)^+ \chi_{E^j}^+ - (\Phi^j)^-\chi_{E^j}^-\big] \cdot \nu_\Sigma\, d\Ha^{1} 
- \sum_{j=1}^m \int_{\partial \Omega_\varepsilon} \chi_{E^{j}}\Phi^j  
\cdot \nu_{\partial \Omega_\varepsilon}\, d\Ha^1
\nonumber\\
+& \sum_{j'=m+1}^{2m} \int_{\Sigma} \big[(\chi_{E^{j'}})^+ 
- (\chi_{E^{j'}})^-\big]\Phi^{j'} \cdot \nu_\Sigma \, d\Ha^1\, ,\label{for}
\end{align}
where $ \nu_{\partial \Omega_\varepsilon}$ is the inner unit normal 
to $\partial \Omega_\varepsilon$.
From now on we will call
\begin{equation*}
a_{h}^+:=(\Phi^h)^+ \cdot \nu_{\Sigma}, 
\quad a_{h}^-:=(\Phi^h)^- \cdot \nu_{\Sigma} \quad \mbox{ for } h=1,\ldots,2m\, .
\end{equation*}
As $\Phi$ is approximately regular (see Definition \ref{app} and Definition \ref{appreg}) one has that
\begin{equation}\label{approx}
a_{j'}^+(x) = a^-_{j'}(x) = (\Phi^{j'} \cdot \nu_\Sigma)(x) \quad \mbox{ for }\Ha^1- a.e\  x\in \Sigma 
\end{equation}
for every $j'=m+1, \ldots, 2m$. Moreover as $E\subset Y_\Sigma$, one can verify that
\begin{equation}\label{t1}
\sum_{j=1}^m a^+_j(x) (\chi_{E^j})^+(x) = \sum_{j'=m+1}^{2m} a^+_{j'}(x) (\chi_{E^{j'}})^+(x)
\end{equation}
and
\begin{equation}\label{t2}
\sum_{j=1}^m a^-_j(x) (\chi_{E^j})^-(x) = \sum_{j'=m+1}^{2m} a^-_{j'}(x) (\chi_{E^{j'}})^-(x)
\end{equation}
for $\Ha^1$- a.e  $x\in \Sigma$.\\
Using \eqref{approx}, \eqref{t1}, \eqref{t2} on Formula \eqref{for} it is easy to see that
\begin{equation*}
\int_{Y_\Sigma} \Phi \cdot D\chi_E 
=- \sum_{j=1}^m \int_{\partial \Omega_\varepsilon} \Phi^j \chi_{E^{j}} 
\cdot \nu_{\partial \Omega_\varepsilon}\, d\Ha^1 \, 
\end{equation*}
and analogously
\begin{equation*}
\int_{Y_\Sigma} \Phi \cdot D\chi_F
=- \sum_{j=1}^m \int_{\partial \Omega_\varepsilon} \Phi^j \chi_{F^{j}} 
\cdot \nu_{\partial \Omega_\varepsilon}\, d\Ha^1\, .
\end{equation*}
As $E, F \in \mathscr{P}_{constr}(Y_\Sigma)$, the functions $\chi_{E^j}$ and $\chi_{F^j}$ 
have same boundary conditions on $\partial \Omega_\varepsilon$ for every $j$. This gives equation \eqref{booh} as we wanted to prove.
\end{proof}

\newpage

\bibliographystyle{abbrv}
\bibliography{cali_c_p}

\end{document}